\documentclass[11pt,letterpaper]{amsart}
\usepackage[matrix,arrow,tips,curve]{xy}

\usepackage{color}
 
\numberwithin{equation}{section}

 \usepackage{color}
 \usepackage{cases}
 \usepackage{amsmath}
 \usepackage[latin1]{inputenc}
\usepackage[T1]{fontenc}
\usepackage{verbatim}
\usepackage{graphicx} 

\def\M{\mathcal{M}}

\def\C{\mathcal{C}}

\def\E{\mathcal{E}}

\def\mcO{\mathcal{O}}

\def\V{\mathcal{V}}

\def\Y{\mathcal{Y}}
\def\W{\mathcal{W}}
\def\Z{\mathcal{Z}}

\def\PP{\mathbb{P}}
\def\CC{\mathbb{C}}

\def\ov#1{\overline{#1}}

\def\Pic{\operatorname{Pic}}

\def\rk{\operatorname{rk}}
\def\ext{{\E}xt}

\def\oc2{\mathcal{O}_{\C_2}}

\newtheorem{thm}{Theorem}[section]
\newtheorem{cor}[thm]{Corollary}
\newtheorem{lem}[thm]{Lemma}
\newtheorem{prop}[thm]{Proposition}
\newtheorem{conj}[thm]{Conjecture}

\theoremstyle{definition}

\newtheorem{rem}{Remark}

\newcommand{\be}{\begin{equation}}
\newcommand{\ee}{\end{equation}}

\begin{document}

\title {Irreducibility of Severi varieties on $K3$ surfaces}

\author{Andrea Bruno}
\address{Andrea Bruno: Dipartimento di Matematica e Fisica, Universit\`a Roma Tre
\hfill \newline\texttt{}  \indent Largo San Leonardo Murialdo 1-00146 Roma, Italy} \hfill \newline\texttt{}
 \email{{\tt andrea.bruno@uniroma3.it}}

\author{Margherita Lelli-Chiesa}
\address{Margherita Lelli-Chiesa: Dipartimento di Matematica e Fisica, Universit\`a Roma Tre
\hfill \newline\texttt{}  \indent Largo San Leonardo Murialdo 1-00146 Roma, Italy} \hfill \newline\texttt{}
 \email{{\tt margherita.lellichiesa@uniroma3.it}}
\begin{abstract}
Let $(Y,L)$ be a general primitively polarized $K3$ surface of genus $g$. For every $0\leq \delta \leq g$ we consider the Severi variety parametrizing integral curves in $|L|$  with exactly $\delta$ nodes as singularities. We prove that its closure in $|L|$ is connected as soon as $\delta\leq g-1$. If $\delta\leq g-4$, we obtain the stronger result that the Severi variety is irreducible, as predicted by a well-known conjecture. The results are obtained by degeneration to Halphen surfaces.
\end{abstract}

\maketitle
\section{Introduction.}

Let $L$ be a polarization on a smooth irreducible projective surface $S$ defined over the field of complex numbers, and denote by $g$ the arithmetic genus of all curves in $|L|$. For any fixed integer $0\leq \delta \leq g$ the {\em Severi variety} of $\delta$-nodal curves in $|L|$ is the  locally closed subscheme of $|L|$ defined as
$$
V_{\delta}(S, L):=\left\{ C\in |L|\textrm{ s.t. }C \textrm{ is integral with exactly }\delta\textrm{ nodes as singularities}   \right\}\!\!;
$$
the same definition applies to singular surfaces $S$ with the further requirement that the curves $C$ lie in the smooth locus of $S$.
These varieties are named after Severi, who introduced them in the case $S=\PP^2$ \cite{Se}, where he proved that they are nonempty and smooth of the expected dimension, namely, $\dim |L|-\delta$. Severi also claimed that they are irreducible, but a rigorous proof of this fact was accomplished only some sixty years later by Harris \cite{Ha}. Since then, Severi varieties were thoroughly investigated for many types of surfaces, in particular as regards their nonemptiness, their local geometry and their irreducibility; the last issue became known as the {\it Severi problem}. Nonemptiness has been established in many cases, as for instance $K3$ surfaces \cite{MM, Ch1}, abelian surfaces \cite{KLM,KL}, Enriques surfaces \cite{CDGK}.
As concerns their local geometry, Severi varieties behave well on rational surfaces and surfaces of Kodaira dimension $0$, while on surfaces of general type wild unexpected phenomena occur, as highlighted in \cite{CS,CC}. 

On the other hand, very little is known about the global geometry of Severi varieties even for surfaces of non-maximal Kodaira dimension. In particular, the Severi problem proves very challenging  and has been solved in very few cases: for Hirzebruch surfaces  by Tyomkin \cite{Tyo}, for Del Pezzo surfaces  in the case of rational curves (that is, for maximal $\delta$) by Testa \cite{Tes}, while partial results for blow-ups of the projective plane are due to Greuel-Lossen-Shustin \cite{GLS}. In recent times, many papers focused on the case of toric surfaces \cite{Bo,LT}, and Zahariuc \cite{Za} worked out the Severi problem for a general abelian surface with a polarization of any primitive type. 

A vast literature is devoted to the case of $K3$ surfaces, motivated by the following  well-known folklore conjecture.
\begin{conj}\label{uno}
Let $(Y,L)$ be a general polarized $K3$ surface of genus $g\geq 2$. Then, for any fixed $0\leq \delta\leq g-1$, the Severi variety $V_{\delta}(Y,L)$ is irreducible.
\end{conj} 
We recall that $\dim |L|=g$. The constraint $\delta \leq g-1$ is necessary because it is well-known that the linear system $|L|$ contains finitely many rational curves: since this number is computed by the Yau-Zaslow formula \cite{Be} and is different from $1$, the Severi variety $V_\delta(Y,L)$ is definitely reducible for $\delta=g$. Quite surprisingly, the above conjecture has remained open until now, despite numerous attempts.  We will prove it for primitive linear systems as soon as $\delta\leq g-4$ and $g\geq 5$.
\begin{thm}\label{andreamarghe}
Let $(Y,L)$ be a general primitively polarized $K3$ surface of genus $g\geq 2$. Then the following hold:
\begin{enumerate}
\item for every $0\leq \delta\leq g-1$, the closure of the Severi variety $\ov{V_{\delta}(Y,L)}\subset |L|$ is connected;
\item if $g\geq 5$ and $0\leq \delta\leq g-4$, the Severi variety $V_{\delta}(Y,L)$ is irreducible.
\end{enumerate}
\end{thm}
Previous results in the literature due to Keilen \cite{Kei}, Kemeny \cite{Ke}, Ciliberto-Dedieu \cite{CD2}, Dedieu \cite{De2} only concerned cases where $\delta$ is small with respect to the arithmetic genus $g$ (roughly bounded by $g/4$) and it was clear that they cannot be further improved with similar proof techniques. A weaker form of the conjecture concerning the so-called {\em universal Severi variety} $\mathcal{V}_{g,\delta}$ was considered more approachable. Let $\mathcal F_g$ be the irreducible $19$-dimensional moduli stack of genus $g$ primitively polarized $K3$ surfaces. The stack $\mathcal{V}_{g,\delta}$ is smooth of pure dimension $19+g-\delta$ and admits a morphism $\phi_{g,\delta}:\mathcal{V}_{g,\delta}\to \mathcal{F}_g^\circ$ to a suitable open substack $\mathcal{F}_g^\circ$ of $\mathcal{F}_g$ whose fiber over a general point $(Y,L)\in \mathcal F_g$ equals the Severi variety $V_{\delta}(Y,L)$. \begin{conj}\label{universal}
For every $0\leq \delta\leq g$, the universal Severi variety $\mathcal{V}_{g,\delta}$ is irreducible.
\end{conj}
This prediction makes perfect sense even for $\delta=g$, when it becomes a question on the monodromy of the finite morphism $\phi_{g,g}$. It is related to the non-existence of self-rational maps of degree $>1$ on a general $K3$ surface in $\mathcal F_g$, which was predicted by Dedieu in \cite{De1} and achieved by Chen in \cite{Ch4}. Conjecture \ref{universal} was proved by Ciliberto-Dedieu \cite{CD} for $2\leq g\leq 11$ and $g\neq 10$, which is exactly the range where a general genus $g$ curve lies on a $K3$ surface. We remark that, since the morphism $\phi_{g,\delta}$ is known to be smooth and dominant on all components of $\mathcal{V}_{g,\delta}$ for every $\delta$ \cite{FKPS}, Conjecture \ref{uno} implies Conjecture \ref{universal} for every $0\leq \delta\leq g-1$. In particular, the following result comes straightforward from Theorem \ref{andreamarghe}.
\begin{cor}\label{viva}
For every $g\geq 5$ and every $0\leq \delta\leq g-4$ the universal Severi variety $\mathcal{V}_{g,\delta}$ is irreducible.
\end{cor}
The assumption $\delta\leq g-4$ in Theorem \ref{andreamarghe}(2) and in Corollary \ref{viva} is due to proof technique, and is only used in the proof of Theorem \ref{local}. However, there is no evidence for the existence of counterexamples to Conjecture \ref{uno} in the remaining cases $g-3\leq \delta\leq g-1$.

\subsection{Strategy and organization of the paper}Theorem \ref{andreamarghe} is  proved by dege\-ne\-ration to a so-called {\em Halphen surface} $\ov S_g\subset \PP^g$, which has an elliptic singularity and is limit of primitively embedded $K3$ surfaces of genus $g$. These surfaces, introduced in \cite{CD}, appeared in the chara\-cte\-ri\-zation of hyperplane sections of $K3$ surfaces accomplished by Arbarello-Bruno-Sernesi \cite{ABS}, and were first exploited in \cite{ABFS} and then in \cite{AB,FT}. We recall their construction. Let $S$ be the  blow-up of $\PP^2$ at $9$ general points and denote by $|L_g|$ the $g$-dimensional linear system on $S$ parametrizing the strict transforms of plane curves of degree $3g$ having multiplicity $g$ at the first $8$ points that we have blown up and multiplicity $g-1$ at the last one; these are called {\em Du Val curves} of genus $g$ after Du Val, who first considered them \cite{DV}. The surface $\ov S_g$ is realized as the closure in $\mathbb P^g$  of the rational map $S\dashrightarrow \PP^g$ defined by $|L_g|$. In particular, Severi varieties of nodal hyperplane curves on $\ov S_g$ are linked to Severi varieties $V_{\delta} (S, L_g)$ on $S$. A major advantage is that the surface $S$ possesses polarizations $L_g$ for every genus $g\geq 2$ and is thus the right environment where to perform some sort of induction. In Section \ref{due}, after recalling the main features of Halphen surfaces, we show that the results known for Severi varieties on a general $K3$ surface of genus $g$ still hold true for $V_\delta(S,L_g)$. In particular, Chen's proof of the density of Severi varieties in any equigeneric locus on a general polarized $K3$ surface \cite{Ch2,Ch3} works with essentially no change in the context of Halphen surfaces. Moreover, the irreducibility of $V_\delta(S,L_g)$ is easily obtained when $\delta$ is small with respect to $g$: this is the basis for our induction.  

In Section \ref{tre}, we show that for any fixed integer $k\geq 1$ the linear system $|L_g|$ sits as a linear space of codimension $k$ inside of $|L_{g+k}|$. However, the subspaces $|L_{g+k-j}|\subset |L_{g+k}|$ for $j\geq 2$ have excess intersection with the Severi varieties $\overline{V_\delta(S,L_{g+k})}$, as follows from the following equality: 
$$
|L_{g+k-1}|\cap \overline{V_\delta(S,L_{g+k}) }=\bigcup_{h=0}^\delta  \overline{V_{\delta-h}(S,L_{g+k-h-1}) }.
$$
To circumvent this problem, we perform a sequence of blow-ups with smooth centers
$ \widetilde {|L_{g+k}|} \to |L_{g+k}|$.
Denoting by $\widetilde{V_\delta(S,L_{g+k}) }$ the strict transform of $\overline{V_\delta(S,L_{g+k}) }$, we prove in Proposition \ref{key1} that $\widetilde{V_\delta(S,L_{g}) }$ is isomorphic to the intersection of $\widetilde{V_\delta(S,L_{g+k}) }$ with the strict transform of $|L_{g+k-1}|$ and the exceptional divisors. We then construct a surjective map 
$$\psi:\widetilde{V_\delta(S,L_{g+k}) }\longrightarrow\widetilde{\mathbb P^{k}}$$
where $\widetilde{\mathbb P^{k}}$ is obtained from $\mathbb P^{k}$ again by a sequence of blow-ups. Lemma \ref{blow} and Theorem \ref{main} prove that $\widetilde{V_\delta(S,L_{g}) }$ is isomorphic a fiber of $\psi$ and that $\psi$ admits a section. By a standard argument using Stein factorization and Zariski's Main Theorem, we conclude that $\widetilde{V_\delta(S,L_{g}) }$ is connected and thus the same holds true for $\overline{V_\delta(S,L_{g}) }$.

In order to deduce connectedness for a general $K3$ surface, we consider a stable type II degeneration $Y_0:=S\cup_J S'$ constructed by appropriately gluing two surfaces $S,S'$, which are both a blow-up of $\PP^2$ at $9$ general points as above, have isomorphic anticanonical divisor $J$ and satisfy $N_{J/S}\simeq N_{J/S'}^\vee$. However the limit on $Y_0$ of a relative genus $g$ polarization on a family of $K3$ surfaces degenerating to it is not unique (for instance, one of such limits contracts $S'$ and maps $S$ to $\ov S_g\subset \mathbb P^g$). "Good limits" on $Y_0$ of moduli spaces of stable maps and Severi varieties on general fibers of the family are obtained by applying the theory of {\em expanded degenerations} and moduli stacks of stable maps $\M_{g-\delta}(Y_0^\mathrm{exp},L)$ to such expansions introduced by Jun Li in \cite{Li1,Li2}. The theory of good degenerations of relative Hilbert schemes developed in \cite{LW} is used to define an expanded linear system $|L_g|^{\mathrm{exp}}$. We recall that an expanded degeneration of $Y_0$ is a semistable model 
$$S\cup_J R\cup_J\ldots\cup_JR\cup_J S'$$
obtained from $Y_0$ inserting a chain of ruled surfaces $R:=\PP(\mathcal O_J\oplus  N_{J/S})$ over $J$. Points of $|L_g|^{\mathrm{exp}}$ parametrize curves that live in some expansion of $Y_0$ and have no components in its singular locus. Analogously, the moduli stack $\M_{g-\delta}(Y_0^\mathrm{exp},L)$ parametrizes stable maps to some expansion of $Y_0$ mapping no component of the domain curve to the singular locus of the target expansion. We let $\overline{\mathcal V_\delta(Y_0^\mathrm{exp},L)}$ be the image in $|L_g|^{\mathrm{exp}}$ of the semi-normalization of the substack of $\M_{g-\delta}(Y_0^\mathrm{exp},L)$ parametrizing smoothable maps. We exploit Li's decomposition (also used in \cite{MPT}) of $\M_{g-\delta}(Y_0^\mathrm{exp},L)$ as a non-disjoint union
$$\M_{g-\delta}(Y_0^\mathrm{exp},L)=\bigcup_{\substack{g_1+g_2=g\\h_1+h_2=g-\delta}}\M_{h_1}(S/J,L_{g_1})\times \M_{h_2}(S'/J,L_{g_2}'),$$
where $\M_{h}(S/J,L_{g})$ stands for the moduli stack of stable relative maps to expanded degenerations $S\cup_JR\cup_J\cdots\cup_JR$ of $(S,J)$ with multiplicity $1$ along the relative divisor $J$. A similar decomposition 
\begin{equation*}
\overline{\mathcal V_\delta(Y_0^\mathrm{exp},L)}=\bigcup_{\substack{g_1+g_2=g\\\delta_1+\delta_2=\delta}}\overline{\V_{\delta_1}(S/J,L_{g_1})}\times \overline{\V_{\delta_2}(S'/J,L_{g_2}')}
\end{equation*}
is obtained in Lemma \ref{tosto}. To prove that $\overline{\mathcal V_\delta(Y_0^\mathrm{exp},L)}$ is connected, in Proposition \ref{paths} one reduces to showing connectedness of the stack $\overline{\V_{\delta}(S/J,L_g)}$ as soon as $\delta\leq g-1$. This is done by showing that the map $\psi$ mentioned above, along with its section, lift to the the stack $\M_{g+k-\delta}(S/J,L_{g+k})$ (cf. Lemma \ref{federico} and Theorem \ref{expanded}). The advantage of using stable maps is that one also obtains connectedness of the relative normalization $\overline{\V_{\delta}(S/J, L_g)}^n$ of  $\overline{\V_{\delta}(S/J, L_g)}$ along  $\overline{\V_{\delta+1}(S/J, L_g)}$.

Part (1) of Theorem \ref{andreamarghe} is the content of Theorem \ref{connected}; a key point is that any two components of the degeneration $\overline{\mathcal V_\delta(Y_0^\mathrm{exp},L)}$ can be connected through a sequence of components whose pairwise intersection is generically reduced. 

Section \ref{cinque} is then devoted to the proof of part (2). First of all, we show that, if $(S,L)\in \mathcal F_g$ is general and $0\leq \delta\leq g-1$, two irreducible components of the Severi variety $\overline{V_{\delta}(S,L)}$ intersect in codimension $1$, if they meet at all (cf. Propositions \ref{two}, \ref{cohen} and Lemma \ref{lemcz}). This is obtained by realizing $\overline{V_{\delta}(S,L)}\subset |L|$ as the image under a generically finite map of a degeneracy locus in $S^{[\delta]}\times |L|$ and using the fact that degeneracy loci of the expected dimension are Cohen-Macaulay. Knowing that $\overline{V_{\delta}(S,L)}$ is connected, in order to prove its irreducibility it is thus enough to show that the codimension $1$ components of its singular locus cannot contain the intersection of two irreducible components. This holds true for $\overline{V_{\delta+1}(S,L)}\subset \mathrm{Sing}\overline{V_{\delta}(S,L)}$ by the connectedness of the relative normalization of $\overline{V_{\delta}(S,L)}$ along $\overline{V_{\delta+1}(S,L)}$.  Let $W$ be any codimension $1$ component of $\mathrm{Sing}\overline{V_{\delta}(S,L)}$ not contained in $\overline{V_{\delta+1}(S,L)}$. By deformation theory (cf. \cite{CD2} for similar arguments), we show that a general point of $W$ parametrizes either a curve whose singularities consist of (possibly non-transverse) smooth linear branches except at most for one cusp, or a curve whose normalization is hyperelliptic of genus $g-\delta$. In the former case, it turns out that $\overline{V_{\delta}(S,L)}$ is unibranched along $W$. The latter case can be excluded as soon as $W$ has dimension $\geq 3$, or equivalently, $\delta \leq g-4$, because curves in $|L|$ with hyperelliptic normalization of any fixed geometric genus $\geq 2$ are known to move in dimension $2$ (cf. \cite[Rmk. 5.6]{KLM}); this is the only part of the proof where the assumption $\delta \leq g-4$ is used.

\subsection{Preliminaries on Severi varieties on $K3$ surfaces}
We will here collect known pro\-per\-ties of Severi varieties on $K3$ surfaces that are relevant for this paper and will be generalized to Halphen surfaces in Section \ref{due}. Standard deformation theory yields the following result (cf., e.g., \cite[\S 3--4]{DS}):
\begin{prop}\label{known}
Let $(Y,L)$ be a polarized $K3$ surface of genus $g$. For any fixed integer $0\leq \delta\leq g$ the Severi variety $V_{\delta}(Y,L)$, if nonempty, is smooth of dimension $g-\delta$.
\end{prop} 
Indeed, for any $C\in V_{\delta}(Y,L)$ the projective tangent space to $V_{\delta}(Y,L)$ at $C$ coincides with $\mathbb P(H^0(Y,L\otimes I_N))$, where $N$ is the scheme of nodes of $C$. Furthermore, the nodes of any such curve $C$ can be smoothed independently. Therefore, the nonemptiness of $V_{\delta}(Y,L)$ for every $\delta$ reduces to the existence in the linear system $|L|$ of a nodal rational curve. This was achieved by Mori-Mukai for a general primitively polarized $K3$ surface, and was then generalized by Chen to non primitive polarizations.
\begin{thm}[\cite{MM, Ch1}]\label{nonempty}
Let $(Y,L)$ be a general $K3$ surface of genus $g$. For any fixed integer $0\leq \delta\leq g$, the Severi variety $V_{\delta}(Y,L)$ is nonempty.
\end{thm}
For primitive polarizations, Chen obtained the following much stronger result: 
 \begin{thm}[\cite{Ch2}]\label{rational}
 Let $(Y,L)$ be a general primitively polarized $K3$ surface of genus $g$. Then, all rational curves in the linear system $|L|$ are nodal. 
 \end{thm}
The above result is deeply linked to the natural question whether every curve in $|L|$ can be deformed to a nodal curve having the same geometric genus. A positive answer is again due to Chen and, defining the {\em equigeneric locus}
$$V^h(Y,L):=\left\{ C\in |L|\,\textrm{ s.t. }\, C \textrm{ is integral of geometric genus }h\right\}$$
for every $0\leq h\leq g$, it can be phrased in the following way. 
\begin{thm}[\cite{Ch3}]\label{dense}
Let $(Y,L)$ be a general primitively polarized $K3$ surface of genus $g$. Then, for every $0\leq \delta\leq g$, the Severi variety $V_{\delta}(Y,L)$ and the equigeneric locus $V^{g-\delta}(Y,L)$ have the same closure in $|L|$. 
\end{thm}
\vspace{0.3cm}
\noindent\textbf{Acknowledgements:} We are especially grateful to Xi Chen and Adrian Zahariuc for pointing out a mistake in a previous version of this paper and for suggesting Proposition \ref{cz}.  We also thank Thomas Dedieu and Edoardo Sernesi for numerous valuable conversations on the topic. We were supported by the Italian PRIN-2017  "Moduli Theory and Birational Classification" and by GNSAGA.

\section{Halphen surfaces and their Severi varieties}\label{due}
Let $S$ be the blow-up of $ \PP^2$ at nine general points
 $p_1,\dots, p_{9}$ 
and denote by $E_1,\dots, E_{9}$ the exceptional curves of this blow-up. As the points $p_i$ are general, there exists a unique plane cubic passing through the $p_i$'s, whose strict transform on $S$ we denote by $J$. Hence, $J$ is the only anticanonical divisor on $S$ and satisfies
$$
J\sim -K_{S}\sim 3l-E_1-\cdots -E_{9},
$$
where $\ell$  is the strict transform of a line in $\PP^2$.
For any fixed $g\geq 1$, let $C$ be the strict transform on $S$ of a so-called {\em Du Val curve of genus $g$}, that is, a plane curve of degree $3g$ having points of multiplicity $g$ at $p_1,\ldots,p_8$ and a point of multiplicity $g-1$ at $p_9$:
$$
C\sim3g\ell-gE_1-\cdots-gE_{8}-(g-1)E_{9}.
$$
Defining $L_g:=\mathcal O_{S}(C)\in \Pic(S)$, 
the linear system $|L_g|$ has dimension $g$ and its general element is a smooth irreducible curve of genus $g$. Since $C\cdot J=1$,  every irreducible curve $C\in |L_g|$ intersects $J$ at the same point, that we denote by $p_{10}(g)$. It turns out that $p_{10}(g)$ is the only base point of $|L_g|$ (cf. \cite[Lem. 2.4]{ABFS}) and is uniquely determined by the condition
$$gp_1+\ldots+gp_8+(g-1)p_9+p_{10}(g)\in |\mathcal O_{J}(3g\ell)|.$$
We will sometimes use the notation $L_0:=E_9$ and $p_{10}(0)=p_9$.

Let $\sigma: \tilde S\longrightarrow S$ be the blow-up of $S$ at $p_{10}(g)$. We still denote by $E_1,\dots, E_{9}$ the inverse images under $\sigma$ of the exceptional curves on $S$
and by $E_{10}$ the exceptional divisor of  $\sigma$. Let $\tilde J$ be the strict transform of $J$ and $\tilde C$ be the strict transform of  a curve $C\in |L_g|$. The following relations hold on $\tilde S$:

\be\label{dv}
\aligned
-&K_{\tilde S}\sim \tilde J\sim 3\ell-E_1-\cdots-E_{10},\\
&\tilde C\sim 3g \ell-gE_1-\cdots-gE_{8}-(g-1)E_{9}-E_{10}\,,\\
&\tilde C\cdot \tilde J=0\,.
\endaligned
\ee

The line bundle $\tilde L_g:=\mathcal O_{\tilde S}(\tilde C)$  is base-point-free \cite[Lem. 2.4]{ABFS} and thus defines a morphism from $\tilde S$ to a surface $\ov S_g\subset \PP^g$ having trivial dualizing sheaf, canonical hyperplane sections and a single elliptic
singularity $o$  resulting from the contraction of $\tilde J$. As in \cite{AB},  we call such a surface $\ov S_g\subset\PP^g$ a polarized Halphen surface of genus $g$. A general hyperplane section of $\ov S_g$ is a smooth irreducible curve of genus $g$ \cite[Lem. 2.4]{ABFS}, while a general hyperplane section of $\ov S_g$ passing through $o$ has a cusp at $o$. The following result is due to Arbarello-Bruno-Sernesi:

\begin{prop}[\cite{ABS}, Cor. 10.5]\label{duval}
If the points $p_1, \ldots, p_9$ are general, the surface $\ov S_g$ is the limit of smooth $K3$ surfaces in $\PP^g$. 
\end{prop}

Halphen surfaces $\ov S_g$ as above share some common behaviour with $K3$ surfaces of Picard rank $1$. This depends on the following property, firstly exploited by Arbarello-Bruno-Farkas-Sacc\`a
\cite{ABFS}.
\begin{lem}\label{decomposition}
If the points $p_1, \ldots, p_9$ are general, for any fixed integer $g\geq 1$ the only possible decompositions of $L_g$ into two effective line bundles are of the form $$L_g\simeq\mathcal O_{S}(kJ)\otimes L_{g-k}$$ for some $0\leq k\leq g-1$.
\end{lem}
\begin{proof}
By choosing $p_1, \ldots, p_9$ general, we may assume that $S$ contains no ($-2$)-curves and that $h^0(S, \mathcal O_{S}(kJ))=1$ for every $k\geq 1$; in other words, $p_1, \ldots, p_9$ are chosen $k$-general in the sense of \cite[Def. 2.2]{AB} for every $k\geq 1$ (cf. also \cite{CD}). Let $L_g\simeq N\otimes M$ be a decomposition into two effective line bundles $N,M\in \mathrm{Pic}(S)$. Since $c_1(L_g)\cdot J=1$ and $(J)^2=0$, possibly by exchanging $N$ and $M$ we obtain $J\cdot c_1(N)=0$ and $J\cdot c_1(M)=1$. The statement thus follows by a theorem of Nagata (\cite[Prop. 2.3]{ABFS}), ensuring that under the genericity assumption the only effective divisors having vanishing intersection with $J$ are the nonnegative multiples of $J$.
\end{proof}

The above result was used by Arbarello-Bruno-Farkas-Sacc\`a in order to prove the following analogue of Lazarsfeld's Theorem.
\begin{thm}[\cite{ABFS}, Thm. 4.4]\label{petri}
If the points $p_1, \ldots, p_9$ are general, then a general curve $C\in |L_g|$ satisfies Petri's Theorem and all irreducible nodal curves in $|L_g|$ satisfy the Brill-Noether Theorem.
\end{thm}

We now investigate Severi varieties $V_\delta(S, L_g)$ and equigeneric loci $V^{h}(S,L_g)$ on $S$. We recall that the normalization $\tilde{V}^{h}(S,L_g) $ of $V^{h}(S,L_g)$ admits a universal family $\mathcal C\to \tilde{V}^{h}(S,L_g) $ together with a simultaneous resolution of singularities $\tilde{\mathcal C}\to\mathcal C$ (cf. \cite[I, Thm. 1.3.2]{Te} and also \cite[Thm. 1.5]{DS}). This implies the existence of an \'etale cover $W\to \tilde{V}^{h}(S,L_g) $ along with a generically injective morphism $w:W\to M_h(S,L_g)$ to the coarse moduli space of genus $h$ stable maps  in $|L_g|$. The image of $w$ consists of the irreducible components of $M_h(S,L_g)$ parametrizing stable maps which are smoothable, that is, can be deformed to a map from a nonsingular curve, birational to its image (cf. \cite{Va}). We denote by $M_h(S,L_g)^{\mathrm{sm}}$ the closure  in $M_h(S,L_g)$ of the image of $w$.

Viceversa, by \cite[I.6]{Ko1} the semi-normalization $\tilde{M}_h(S,L_g)^{\mathrm{sm}}$ of $M_h(S,L_g)^{\mathrm{sm}}$ admits a morphism 
\begin{equation}\label{semi}
\mu: \tilde{M}_h(S,L_g)^{\mathrm{sm}}\to \overline{V^{h}(S,L_g)}\subset |L_g|,
\end{equation} that maps a stable map $f:C\to S$ to its image $f(C)$.
\begin{prop}\label{dimension}
The following hold true:
\begin{itemize}
\item[(i)] For every $0\leq \delta \leq g$ the Severi variety $V_\delta(S, L_g)$ is nonempty and smooth of dimension $g-\delta$.
\item[(ii)] For every $0\leq h\leq g$ the equigeneric locus $V^h(S,L_g)$  and $M_h(S,L_g)^{\mathrm{sm}}$ have pure dimension $ h$. 
\item[(iii)] For every $0\leq h\leq g$ a general point $C$ in any irreducible component of $V^h(S,L_g)$ is immersed\footnote{A curve is called immersed if the differential of its normalization map is everywhere injective.}; equivalently, a general map $f$ in any irreducible component of $M_h(S,L_g)^{\mathrm{sm}}$ is unramified. In particular, both $V^h(S,L_g)$ and $M_h(S,L_g)^{\mathrm{sm}}$ are generically reduced.
\end{itemize}
\end{prop}
\begin{proof}
We recall that the expected dimension $V_\delta(S, L_g)$ is $g-\delta$. The nonemptiness statement in (i) follows from \cite[Thm. B]{GLS}. By standard deformation theory, the projective tangent space to $V_\delta(S,L_{g})$ at a point $C$ is isomorphic to $\mathbb P(H^0(S,L_{g}\otimes I_N))$, where $N$ is the scheme of nodes of $C$. Hence, $V_\delta(S, L_g)$ is smooth at $C$ of dimension $g-\delta$ if and only if $h^0(L_g\otimes I_N)=g+1-\delta$, or equivalently, $h^1(L_g\otimes I_N)=0$. This vanishing can be easily deduced by the short exact sequence
$$
0\to \mcO_{S}\to L_g\otimes I_N\to \omega_C(p_{10}(g))\otimes I_N\to 0,
$$
using the isomorphism $\omega_C(p_{10}(g))\otimes I_N\simeq \nu_*\omega_{\tilde C}(p)$, where $\nu:\tilde C\to C$ is the normalization map and $p=\nu^{-1}(p_{10}(g))$.

As concerns part (ii), let $C$ be a general point in any irreducible component $V$ of $V^h(S,L_g)$ and let $f:\tilde C\to S$ be the stable map defined as the composition of the normalization map $\nu:\tilde C\to C$ with the inclusion $C\subset S$. The discussion above the statement of this proposition yields that $\dim_{[C]}V=\dim_{[f]} M_h(S,L_g)$ and, by  standard deformation theory,  the latter is bounded below by $\chi(N_f)$, where $N_f$ is the normal sheaf to $f$ defined by the short exact sequence 
$$
0\to T_{\tilde C}\to f^*T_{S}\to N_f\to 0.
$$ 
It is then easy to check that $\chi(N_f)=\chi(\omega_{\tilde C}(p))=h$ and thus $\dim V\geq h$.

In order to prove equality, we apply a result by Arbarello and Cornalba \cite[p. 26]{AC} as in \cite[proof of Thm. 2.8]{DS}  getting
$$
 \dim V=\dim T_{[C]}V\leq h^0(\tilde C, \overline N_f)=h^0(\omega_{\tilde C}(p-R))\leq h,
$$
where $\overline N_f$ denotes the quotient of $N_f$ by its torsion subsheaf, which coincides with the zero divisor  $R\subset \tilde C$ of the differential of $f$. Since $\omega_{\tilde C}(p)$ is globally generated off $p$ and $p_{10}(g)=f(p)$ is a smooth point of $C$, we conclude that $\dim V=h$ (thus getting (ii)) and $R=0$. Hence, $C$ is immersed and this yields (iii) because $T_{[f]}M_h(S,L_g)=h^0(N_f)=h$ and $\mu$ is an isomorphism locally around $[f]$.\end{proof}

It is natural to ask whether the closure in $|L_g|$ of the Severi variety $V_\delta(S, L_g)$ coincides with that of the equigeneric locus $V^{g-\delta}(S, L_g)$, as it happens on a general $K3$ surface. The following result generalizes Theorem \ref{dense} to our setting.
\begin{prop}\label{nodal}
If the points $p_1, \ldots, p_9$ are general, then for every $g\geq 1$ and $0\leq \delta\leq g$ one has the equality  $$\overline{V_\delta(S, L_g)}=\overline{V^{g-\delta}(S, L_g)}$$ in the linear system $|L_g|$. 
\end{prop}
\begin{proof}
We follow Chen's proof of the analogous result for a general genus $g$ polarized $K3$ surface \cite[Cor. 1.2]{Ch3}. Let $V$ be any irreducible component of the equigeneric locus $\overline{V^h(S, L_g)}$ with $0\leq h\leq g$. In order to prove that a general point of $V$ parametrizes a nodal curve, it is enough to show that $V$ contains a component of $V^{h-1}(S, L_g)$ as soon as $h\geq 1$, and that all rational curves in $|L_g|$ are nodal. Both the statements were proved  for a general genus $g$ polarized $K3$ surface by Chen (in \cite[Thm. 1.1]{Ch3} and \cite[Thm. 1.1]{Ch2}, respectively), by specialization to a so-called Bryan-Leung $K3$ surface, that is, a $K3$ surface  $X_0$ admitting an elliptic fibration $\pi:X_0\to \PP^1$ with a section $s$ and $24$ nodal singular fibers. If $f$ is a fiber, the line bundle $L_0:=\mathcal O_{X_0}(s+gf)$ is a genus $g$ polarization on $X_0$ and every element in $|L_0|$  is completely reducible, that is, it is union of $s$ and $g$ fibers of $\pi$.

We now exhibit a limit of our surfaces $S$ that appeared in \cite[\S 4.1]{ABFS} and is very similar to a Bryan-Leung $K3$ surface. By specializing the points $p_1, \ldots, p_9\in \mathbb P^2$ to the base locus of a general pencil of plane cubics, the surface $S$ specializes to a rational elliptic surface $q: S_0\to \mathbb P^1$; the fibers of $q$ are the anticanonical divisors of $S_0$ and thus $q$ admits precisely $12$ nodal singular fibers. It is easy to verify that on $S_0$ the exceptional divisor $E_9$  becomes a section of $q$ and every element in the linear system $|L_g|$ is the union of  $E_9$ with $g$ fibers of $q$. Chen's proof of \cite[Thm. 1.1]{Ch3} works in our setting with no change, yielding that on $S_0$ (and thus on a general $S$) every component of $\overline{V^h(S, L_g)}$ with $h\geq 1$ contains a component of $V^{h-1}(S, L_g)$. Also the proofs in \cite{Ch2} still work if, instead of  a family of $K3$ surfaces whose central fiber is a Bryan-Leung  $K3$ surface, one considers a family of surfaces $\mathcal S\to \Delta$ whose general fibers are general $S$ and whose central fiber is $S_0$. The only difference that is worth remarking concerns \cite[Prop. 2.1]{Ch2}, whose proof becomes even simpler in our case because every vector of the space $H^1(T_{S_0})$ parametrizing first order deformation of $S_0$ can be realized as the Kodaira-Spencer class of a projective family $\mathcal S$.
\end{proof}

The following result is a generalization of \cite{Ke,CD2}, ensuring irreducibility of Severi varieties in $|L_g|$ when $\delta$ is small with respect to $g$.
\begin{prop}\label{irreducible}
If $\delta\leq\frac{1}{6}g-\frac{1}{12}$, then $V_\delta(S,L_g)$ is irreducible.
\end{prop}
\begin{proof}
  Let $U_{\delta}\subset S^{[\delta]}$ be the open subset parametrizing $0$-dimensional subschemes consisting of $\delta$ distinct points none of which lies on $J$. The nodes of any curve $C\in V_\delta(S,L_g)$ define a point in $U_{\delta}$ because they all lie outside of $J$ as $C\cdot J=1$. As in \cite[App. A, proof of Thm. A.0.6]{Ke}, the Severi variety is an open subset of a projective bundle over $U_{\delta}$ as soon as $H^1(L_g\otimes I_z^2)=0$ for all $z\in U_{\delta}$. We will show that $H^1(L_g\otimes I_w)=0$ for every $w\in S^{[3\delta]}$ whose support is disjoint from $J$. By contradiction, if this is not the case, up to replacing $w$ with a subscheme of length $d\leq 3\delta$, we may assume that $h^1(L_g\otimes I_w)=1$ (use \cite[Lem. 1.2]{BS}) and $h^1(L_g\otimes I_{w'})=0$ for every proper subscheme $w'$ of $w$ (that is, $w$ is $L_g$-stable in the sense of Tyurin \cite[Def. 1.2]{Ty}). By \cite[Lem. 1.2]{Ty} there exists a rank $2$ vector bundle $E$ fitting into an extension
\begin{equation}\label{extension}
0\to \mathcal O_S\to E\to L_{g-1}\otimes I_w\to 0,
\end{equation}
where we have used that $L_g\otimes K_{S}\simeq L_{g-1}$. Since $c_1(E)=c_1(L_{g-1})$ and $c_2(E)=d\leq3\delta$, the Riemann-Roch formula yields
$$
\chi(E\otimes E^\vee)=c_1(E)^2-4c_2(E)+4\chi(\mathcal O_{S})=2g-3-4d+4\geq2g-12\delta+1\geq 2,
$$
and thus either $h^0(E\otimes E^\vee)\geq 2$ or $h^2(E\otimes E^\vee)=h^0(E\otimes E^\vee\otimes K_{S})\geq 1$. In both cases, $E$ is not $\mu_{L_{g-1}}$- stable and thus sits in a destabilizing short exact sequence
\begin{equation}\label{destabilizing}
0\to N\to E\to M\otimes I_{\xi}\to 0,
\end{equation}
 where $\xi\subset S$ is a $0$-dimensional subscheme and $N,M\in \Pic(S)$ satisfy $$\mu_{L_{g-1}}(N)\geq \mu_{L_{g-1}}(E)=\frac{2g-3}{2}\geq \mu_{L_{g-1}}(M).$$ In particular, one gets $h^0(N^\vee)=0$. We will use short exact sequence \eqref{destabilizing} and Lemma \ref{decomposition} to reach a contradiction.  As in \cite[Lem 3.6]{Kn}, by tensoring \eqref{extension} with $N^\vee$ and taking global sections, one obtains $h^0(M\otimes I_w)>0$ and thus $M$ possesses a global section vanishing along a divisor that contains $w$; in particular, $M$ is effective and $M\not\simeq \mathcal O(kJ)$ for any $k\geq 0$. If \eqref{destabilizing} splits, the same holds true for $N$ by inverting the roles of $N$ and $M$. Since $c_1(L_{g-1})=c_1(N)+c_1(M)$, this would contradict Lemma \ref{decomposition}: we conclude that \eqref{destabilizing} does not split. 
 
 In the case where $h^0(E\otimes E^\vee)\geq 2$, by standard computations (cf., e.g., \cite[Lem. 3.4]{AF}) one concludes that $h^0(N\otimes M^\vee)>0$ and thus $N$ is effective, too. By Lemma \ref{decomposition}, we get that $N\simeq \mathcal{O}(kJ)$ and $M\simeq L_{g-k}'$ for some $0\leq k\leq g-1$ and this contradicts the inequalities on the slopes. 
 
 In order to arrive at the same conclusion in the case where $h^0(E\otimes E^\vee\otimes K_{S})\geq 1$, we tensor \eqref{extension} with $K_{S}$ and then apply $\mathrm{Hom}(E,-)$ in order to get $$h^0(E\otimes E^\vee\otimes K_{S})\leq \dim \mathrm{Hom}(E,N\otimes K_{S})+ \dim \mathrm{Hom}(E,M\otimes K_{S}\otimes I_\xi).$$
By applying $\mathrm{Hom}(-,M\otimes K_{S}\otimes I_\xi)$ to \eqref{extension} and using the fact that $\mathrm{Hom}(N,M\otimes K_{S}\otimes I_\xi)=0$ as $\mu_{L_{g-1}}(N)>\mu_{L_{g-1}}(M\otimes K_{S})$, one obtains that $\mathrm{Hom}(E,M\otimes K_{S}\otimes I_\xi)=0$. Analogously, applying $\mathrm{Hom}(-,N\otimes K_{S})$  to \eqref{extension}, one shows that $1\leq \mathrm{Hom}(E,N\otimes K_{S})\simeq H^0(M^\vee\otimes N\otimes K_{S})$; hence, $N$ is effective yielding the same contradiction as above.
\end{proof}
The following result controls the intersection of two irreducible components of $\overline{V_\delta(S,L_g)}$.
\begin{prop}\label{cohen}
Fix $g\geq 2$ and $0\leq \delta\leq g-1$. Let $V$ and $W$ be two intersecting components of $\overline{V_\delta(S,L_g)}$. Then every irreducible component of $V\cap W$ not contained in $|L_{g-1}|$ has pure codimension $1$ and is generically reduced. 
\end{prop}
\begin{proof}
Set $U:=|L_g|\setminus |L_{g-1}|\subset |L_g|$ and consider the incidence variety 
\begin{equation}\label{incidence}
I:=\overline{\left\{ (C,z)\in U\times S^{[\delta]}\,\textrm{ s.t. }C\in |L_g\otimes I_z^2|   \right\}}\subset |L_g|\times S^{[\delta]}.
\end{equation}
We will express $I$ as the degeneracy locus of a map of vector bundles on $|L_g|\times S^{[\delta]}$. Let $p : S \times S^{[\delta]} \longrightarrow S$ and $q: S \times S^{[\delta]} \longrightarrow S^{[\delta]}$ be the projections,
and denote by $\Delta \subset S \times S^{[\delta]}$ the universal subscheme. Let $E:=q_*(p^*L_g)$ denote the vector bundle of rank $g+1$ on $S^{[\delta]}$ whose fiber over any point $z\in S^{[\delta]}$ equals $H^0(S,L_g)$. Let $F:=q_*(p^*L_g|_{2\Delta})$ be the vector bundle  of rank $3\delta$ on $S^{[\delta]}$ whose fiber over a point $z\in S^{[\delta]}$ is the vector space $H^0(S,L_g|_{2z})$, where $2z$ denotes the $0$-dimensional subscheme of $S$ defined by the ideal $I_z^2$. There is a natural map  
$$\phi:E\longrightarrow F$$
of vector bundles on $S^{[\delta]}$.
Note that $|L_g|\times S^{[\delta]}$ is isomorphic to the projective bundle $$\pi:\mathbb P(E)\longrightarrow S^{[\delta]}.$$ 
Denoting by $\mathcal U\subset \pi^*E$ the universal subbundle, we consider the degeneracy locus $D(\widetilde\phi)$ of the map
$$\widetilde\phi: \mathcal U \longrightarrow \pi^*F,$$
 of vector bundles on $\mathbb P(E)\simeq |L_g|\times S^{[\delta]}$ obtained by composing $\pi^*\phi$ with the inclusion of $\mathcal U$ in $\pi^*E$. By construction, the incidence variety $I$ is contained in $D(\widetilde\phi)$ and, if $(C,z)\in D(\widetilde\phi)\setminus I$, then $C\in |L_{g-1}|$. It can be easily checked that the expected dimension of $D(\widetilde\phi)$ equals $g-\delta$. In order to show that $D(\widetilde\phi)$ has the expected dimension along $I$, we consider the projection $t:D(\widetilde\phi)\longrightarrow |L_g|$. If $(C,z)\in D(\widetilde\phi)$, the curve $C$ is singular along $z$. Hence, if $(C,z)\in I$ and $C$ is reduced, then the $\delta$-invariant of $C$ is $\geq \delta$: this implies that $t(I)\subset\overline{V^{g-\delta}(S,L_g)}=\overline{V_\delta(S,L_g)}$, where the equality follows from Proposition \ref{nodal}. On the other hand, if $(C,z)\in D(\widetilde\phi)\setminus I$, then $C\in |L_{g-1}|\subset |L_g|$. We conclude that $I$ consists of the irreducible components of $D(\widetilde\phi)$ whose image is not entirely contained in $|L_{g-1}|$. In particular, a general curve in $t(I)$ is reduced. The following Lemma \ref{lemcz} yields that the locus in $I\setminus (t^{-1}|L_{g-1}|\cap I)$ where the fibers of $t_I:=t|_I$ are not finite has dimension $\leq g-\delta-2$, and thus
 $$g-\delta=\dim \overline{V_\delta(S,L_g)}\geq \dim I,$$ and $I$ consists of irreducible components of $ D(\widetilde\phi)$ that dominate $\overline{V_\delta(S,L_g)}$ and have the expected dimension. 
 In particular, $I$ is locally Cohen-Macaulay (cf. \cite[II, Prop. 4.1]{ACGH}) outside of its intersection with $t^{-1}|L_{g-1}|$. Hence, every irreducible component $I'$ of the intersection of two components of $I$ has codimension $1$ by Hartshorne's Connectedness Theorem (cf. \cite[Thm. 18.12]{Ei} and is generically reduced, unless possibly when $I'$ is contained in $t^{-1}|L_{g-1}|$. Let $Z$ be a component of the intersection of two irreducible components of $\overline{V_\delta(S,L_g)}$ such that $Z$ is not contained in $|L_{g-1}|$. Since any component $I'$ of $t^{-1}(Z)$ has codimension $1$ in $I$ and is generically reduced, a general fiber of the restriction of $t$ to $I'$ is finite by the following Lemma \ref{lemcz} and thus $Z$  has codimension $1$ in $\overline{V_\delta(S,L_g)}$. If a general curve in $Z$ has $\delta$-invariant precisely $\delta$, then the restriction of $t$ to $t^{-1}(Z)$ is birational and we may conclude that $Z$ is generically reduced. If instead a general curve in $Z$ has $\delta$-invariant $>\delta$, by dimensional reasons $Z$ is a component of $\overline{V_{\delta+1}(S,L_g)}=\overline{V^{g-\delta-1}(S,L_g)}$. We recall that $\overline{V_{\delta}(S,L_{g})}=\overline{V^{g-\delta}(S,L_{g})}$ is singular at the points of $\overline{V^{g-\delta-1}(S,L_{g})}$ (cf. \cite{DH}) as, in a neighborhood of a general $C\in V^{g-\delta-1}(S,L_{g})$, the locus $\overline{V^{g-\delta}(S,L_{g})}$ is the union of at most $\delta+1$ sheets corresponding to the partial normalizations of $C$ of arithmetic genus $g-\delta$. In particular, $\overline{V^{g-\delta}(S,L_{g})}$ is generically reduced along $\overline{V^{g-\delta-1}(S,L_{g})}$ as soon as $\overline{V^{g-\delta-1}(S,L_{g})}$ is generically reduced and this holds true in our case by Proposition \ref{dimension}.
 \end{proof}
\begin{lem}\label{lemcz}
Let $I\subset |L_g|\times S^{[\delta]}$ be the incidence variety defined in \eqref{incidence} and let $t_I: I\longrightarrow |L_g|$ be the first projection.
Then, the locus  in $I\setminus t_I^{-1}|L_{g-1}|$ where the fibers of $t_I$ are not finite has dimension $\leq g-\delta-2$.
\end{lem}
\begin{proof}
For any $k\geq 1$, let $Z_{\delta,k}\subset \overline{V_{\delta}(S,L_g)}$ be the locus of irreducible curves $C\in\overline{V_{\delta}(S,L_g)}$ such that, denoting by $\nu:\tilde C\to C$ the normalization of $C$ and by $A_C:=Hom_{\mathcal O_C}(\nu_*\mathcal O_C,\mathcal O_C)$ its adjoint ideal, the subscheme $E_C\subset C$ defined by $A_C$ contains a $k$-dimensional family of subschemes of length-$\delta$; in particular, if $\mathrm{dim}\,t_I^{-1}(t_I(C))\geq k$ then $C\in Z_{\delta,k}$. We will show that $Z_{\delta,k}\subset \overline{V_{\delta+k+2}(S,L_g)}$ for all $0\leq\delta\leq g-1$ and $k\geq 1$ and thus 
\begin{equation}\label{pe}
\dim\,t_I^{-1}(Z_{\delta,k})=\dim\,Z_{\delta,k}+k\leq \dim\,  \overline{V_{\delta+k+2}(S,L_g)}+k=g-\delta-2,
\end{equation} that yields our statement . 

We proceed by induction on $k$. The case $k=1$ amounts to showing that, if $C\in Z_{\delta,1}$, then the $\delta$-invariant $\delta(C)$ of $C$ (i.e., the length of $E_C$) is $\geq \delta+3$; this holds true because any subscheme $\xi_t$ of length $\delta$ contained in $E_C$ corresponds to a partial normalization $\nu_t:\hat{C}_t\to C$ with $p_a(\hat{C}_t)=p_a(C)-\delta$. If $\delta(C)=\delta$, then necesserily $\nu_t=\nu$, and thus $\xi_t=E_C$ is unique. Analogously, if $\delta(C)=\delta+1$, then any such $\hat{C}_t$ is obtained from $\tilde C$  by creating either one node or one cusp at the finitely many points of $\tilde C$ mapping to the singular locus of $C$; hence, the partial normalizations $\nu_t$ (or, equivalently, the subschemes $\xi_t$) are finitely many in this case. The remaining case $\delta(C)=\delta+2$ is treated in the same way using the fact that the only singularity having $\delta$-invariant equal to $2$ are tacnodes, ramphoid cusps and triple points of embedded dimension $3$.
 
We now assume that the inclusion $Z_{\delta,h}\subset \overline{V_{\delta+h+2}(S,L_g)}$ holds for any $0\leq \delta\leq g-1$ and $1\leq h\leq k-1$, and prove it for $h=k\geq 2$. Fix a general $C\in Z_{\delta,k}$, that is, $C$ has a $k$-dimensional family of length-$\delta$ subschemes contained in $E_C$. We will prove that $C$ possesses a ($k-1$)-dimensional family of subschemes of length $\delta +1$; this is enough to conclude because it implies that $C\in Z_{\delta+1,k-1}\subset  \overline{V_{\delta+k+2}(S,L_g)}$, where the inclusion follows from the induction assumption. In order to pass from subschemes of length $\delta$ to subschemes of length $\delta+1$, we consider the nested Hilbert scheme $S^{[\delta,\delta+1]}$ parametrizing pairs $(\xi,\xi')\in S^{[\delta]}\times S^{[\delta+1]}$ such that $\xi\subset \xi'$. This is endowed with two natural morphisms
 \begin{equation*}
\xymatrix{ 
&S^{[\delta,\delta+1]} \ar[ld]_{\phi}     \ar[rd]^{\psi} &\\
S^{[\delta]}\times S &     &S^{[\delta+1]}\times S,
}
\end{equation*}
mapping a pair $(\xi,\xi')$ to $(\xi,x)$ and $(\xi',x)$, respectively, with $x\in S$ being the point where $\xi$ and $\xi'$ differ. As explained in \cite[p.12]{Le}, the dimensions of the fibers of $\phi$ and $\psi$ are related as follows: given $(\xi,\xi')\in S^{[\delta,\delta+1]}$, if $\phi^{-1}(\phi(\xi,\xi'))\simeq\mathbb P^{i-1}$, then $\psi^{-1}(\psi(\xi,\xi'))\simeq \mathbb P^{i'-2}$ for some integer $i'$ satisfying $|i-i'|\leq 1$. Let us consider the $k$-dimensional family $$B:=\{\xi\in S^{[\delta]}\,\,|\,\, \xi\subset E_C\subset C\},$$ the subscheme $$W:=\{(\xi,\xi')\in S^{[\delta,\delta+1]}\,\,|\,\,\xi\subset\xi'\subset E_C\}\subset \phi^{-1}(B\times \mathrm{Supp}(E_C)),$$ and set $B':=\psi(W)\subset \psi(\phi^{-1}(B\times \mathrm{Supp}(E_C)))$. We want to show that $B'$ has dimension $\geq k-1$. Let $(\xi,\xi')\in \phi^{-1}(B\times \mathrm{Supp}(E_C))$ be general, and denote by $i-1$ the dimension of $\phi^{-1}(\phi(\xi,\xi'))$ and by $i'-2$ the dimension of $\psi^{-1}(\psi(\xi,\xi'))$. Since $i-i'\geq -1$, we obtain 
$$\dim\,  \psi(\phi^{-1}(B\times \mathrm{Supp}(E_C)))=\dim B+i-1-(i'-2)\geq k.$$
In order to conclude that $\dim\,B'\geq k-1$, it is thus enough to show that the fiber at $(\xi,x)$ of the restriction $\phi|_W$ has codimension at most $1$ in the fiber $\phi^{-1}(\xi,x)$. We need to recall that $\phi^{-1}(\xi,x)\simeq \mathbb P((I_\xi/m_xI_\xi)^\vee)$ where $m_x$ is the maximal ideal of the point $x$; indeed, any pair $(\xi,\xi')\in \phi^{-1}(\xi,x)$ corresponds to a short exact sequence
$$
0\longrightarrow I_{\xi'}\longrightarrow I_\xi\stackrel{\alpha}{\longrightarrow} \mathcal O_x\longrightarrow 0,
$$        
and thus to a linear map $\alpha_x:I_\xi/m_xI_\xi\to \mathbb C$. Since $\xi\subset E_C$, we have an inclusion $\iota:I_{E_C/S}\to I_\xi$ and a linear map $\iota_x:I_{E_C/S}/m_x I_{E_C/S}\to I_\xi/m_xI_\xi$; the subscheme $\xi'$ is contained in $E_C$ precisely when the composition $\alpha_x\circ \iota_x$ is zero. In order to show that this is a codimension $1$ condition on $\mathbb P((I_\xi/m_xI_\xi)^\vee)$ we prove that $I_{E_C/S}/m_xI_{E_C/S}$ is $1$-dimensional, or equivalently, $E_C$ is contained at most one subscheme of $S$ of length $\delta(C)+1$. Consider the standard short exact sequence of ideals
$$
0\longrightarrow \mathcal O_S(-C)\stackrel{j}{\longrightarrow} I_{E_C/S}{\longrightarrow}A_C\longrightarrow 0,
$$
where $j$ is the moltiplication by the section of $L_g$ defining $C$. Since $C$ is singular along $E_C$ and $x\in\mathrm{Supp}(E_C)$, then the image of $j$ is contained in $m_x I_{E_C/S}$ and we have an isomorphism $I_{E_C/S}/m_xI_{E_C/S}\simeq A_C/m_xA_C$. It is thus enough to verify that $A_C/m_xA_C\simeq \mathbb C$, or equivalently, $E_C$ is contained at most one subscheme  of $C$ of length $\delta(C)+1$.  This holds true because any such subscheme corresponds to a rank $1$ torsion free sheaf on $C$ of the form $\nu_*\mathcal O_{\tilde C}(y)$ for one of the finitely many points $y$ mapping to $x$.
 \end{proof}

\section{Connectedness on Halphen surfaces}\label{tre}
From now on, we will always assume that $S$ is obtained by blowing-up $9$ general points so that the equality $\overline{V_\delta(S, L_g)}=\overline{V^{g-\delta}(S, L_g)}$ holds by Proposition \ref{nodal} for every $g\geq 1$ and $0\leq\delta\leq g$.

For every $g\geq 0$ and $k\geq 1$, we consider the natural injection
$$i_{g,k}: |L_g|\hookrightarrow |L_{g+k}|$$
mapping a curve $C\in |L_g|$ to the divisor $C+kJ\in |L_{g+k}|$.
\begin{lem}\label{inclusion}
The image $i_{g,k}(|L_g|)\subset |L_{g+k}|$ coincides with the codimension $k$ linear subspace $|L_{g+k}\otimes I_x^k|$, where $x\in J$ is a general point. In particular, identifying $|L_{p}|$ with its image in $|L_q|$ under the map $i_{p, q-p}$ for every $q\geq 2$ and $0\leq p\leq q$, we have the following chain of inclusions in $|L_{q}|$:
\begin{equation}\label{chain}
\{E_9\}=|L_0|\subset |L_1|\subset\cdots  \subset |L_{q-1}|\subset  |L_{q}|
\end{equation} 
\end{lem}
\begin{proof}
The inclusion $i_{g,k}(|L_g|)\subset |L_{g+k}\otimes I_x^k|$ is obvious. In order to prove equality, it is enough to show that  $h^0(L_{g+k}\otimes I_x^k)=g+1$. We proceed by induction on $k$. The case $k=1$ is trivial and the induction step follows from the short exact sequences
$$
0\longrightarrow L_{g+k-1}\otimes I_x^{k-1}\longrightarrow L_{g+k}\otimes I_x^k\longrightarrow L_{g+k}\otimes I_x^k|_J\longrightarrow 0,
$$
along with the isomorphism $L_{g+k}\otimes I_x^k|_J\simeq \mcO_J(p_{10}(g+k)-kx)$.
\end{proof}

\subsection{Connectedness of the closure of Severi varieties}
\begin{prop}\label{cz}
For every $g\geq 2$ and $0\leq \delta\leq g-1$, the following equality holds in $|L_{g}|$: 
\begin{equation}\label{fondamentale}
|L_{g-1}|\cap \overline{V_\delta(S,L_{g}) }=\bigcup_{h=0}^\delta  \overline{V_{\delta-h}(S,L_{g-h-1}) }.
\end{equation}
\end{prop}
\begin{proof}
First of all, we verify the inclusion $\supset$ in \eqref{fondamentale} by showing that, if $0\leq h\leq \delta$ and $C$ is general in any irreducible component of  $\overline{V_{\delta-h}(S,L_{g-h-1})}$, the curve $X=C+ (h+1)J\in |L_{g}|$ can be deformed to an irreducible curve in $|L_{g}|$ of geometric genus $g-\delta$ (which thus lies in $\overline{V_\delta(S,L_{g})}$ by Proposition \ref{nodal}). The curve $X$ is the image of a stable map $f:\tilde C\cup_p \tilde J\to S$ of genus $g-\delta$, where $\tilde J$ is a smooth elliptic degree $h+1$ cover of $J$, the curve $\tilde C$ is the normalization of $C$ and has genus $g-\delta-1$, and the gluing point $p$ is mapped to $C\cap J=\{ p_{10}(g-h-1) \}$. We denote by $f_C:=f|_{\tilde C}:\tilde C\to C\subset S$ and by $f_{J}:=f|_{\tilde J}:\tilde J\to J\subset S$ the restrictions of $f$ to $\tilde C$ and $\tilde J$, respectively.  As $f_{J}$ is étale and $C$ is nodal, both $f_{J}$ and $f_C$ are unramified and the same holds true for the map $f$ since $ C$ and $J$ intersect transversally at $p_{10}(g-h-1)$. The normal sheaf $N_f$ sits in the following short exact sequence: 
$$
0\longrightarrow N_f(-p)|_{\tilde{J}}\longrightarrow N_f\longrightarrow N_f|_{\tilde C}\longrightarrow 0.
$$
By \cite[Lem. 2.5]{GHS} we have isomorphisms 
$$N_f|_{\tilde C}\simeq N_{f_C}(p)\simeq \omega_{\tilde C}(2p),$$
and analogously
$$N_f(-p)|_{\tilde{J}}\simeq N_{f_{J}}\simeq f_{J}^*\mathcal O_{J}(J).$$
Since the line bundle $f_{J}^*\mathcal O_{J}(J)$ is non-trivial of degree $0$, we obtain that $h^0(N_f)=h^0(\omega_{\tilde C}(2p))=g-\delta$ and $h^1(N_f)=0$, and thus $f$ defines a smooth point of a ($g-\delta$)-dimensional component of $M_{g-\delta}(S, L_{g})$. However, $f_C$ is an unramified stable map of genus $g-1-\delta$ and thus $\dim_{[f_C]}M_{g-1-\delta}(S, L_{g-h-1})=g-1-\delta$. Analogously, the map $f_{J}$ is rigid in $M_{1}(S, (h+1)J)$. Hence, a general deformation of $f$ parametrizes a stable map from an integral (and smooth by dimensional arguments) curve of genus $g-\delta$. This proves that $f$ is smoothable and thus the existence of the morphism \eqref{semi} yields that $X\in  \overline{V_\delta(S,L_{g}) }$.

It remains to verify the inclusion $\subset$ in \eqref{fondamentale}. Any irreducible component $V$ of  $|L_{g-1}|\cap\overline{V_\delta(S,L_{g})}$ satisfies $\dim V=g-1-\delta$ because no component of $\overline{V_\delta(S,L_{g})}$ is contained in $|L_{g-1}|$ by Proposition \ref{dimension} and $|L_{g-1}|$ is a hyperplane in $|L_{g}|$. Assume now that a general element $X$ of $V$ parametrizes a curve in $|L_{g-h-1}|\setminus |L_{g-h-2}|$ for some $h\geq 0$, that is, $X=C+(h+1)J$ with $C$ irreducible. We need to show that $h\leq \delta$ and $C\in V_{\delta-h}(S,L_{g-h-1})$; by dimensional reasons and Proposition \ref{nodal}, it is enough to check that $C$ has geometric genus $\leq g-1-\delta$. This trivially follows because $X=C+(h+1)J\in \overline{V_\delta(S,L_{g})}=\overline{V^{g-\delta}(S,L_{g})}$ and $J$ is an elliptic curve.

\end{proof}
We fix $k>>0$ so that $\overline{V_\delta(S,L_{g+k}) }$ is irreducible by Proposition \ref{irreducible}. Inside $|L_{g+k}|$ we consider the chain of inclusions provided by Lemma \ref{inclusion}
$$
\{E_9\}=|L_0|\subset |L_1|\subset\cdots  \subset |L_{g+k-1}|\subset  |L_{g+k}|
$$
and choose hyperplanes $H_1,\ldots H_{g+k}\subset  |L_{g+k}|$ such that for each $1\leq i\leq g+k$ the following equality holds:
$$|L_{g+k-i}|=\left(\bigcap_{j=1}^i H_j\right)\cap |L_{g+k}|.$$
By Proposition \ref{cz}, $\overline{V_\delta(S,L_{g})}$ does not coincide with the intersection of  $\overline{V_\delta(S,L_{g+k}) }$ with a linear subspace. In order to circumvent this problem we perform a sequence of blow-ups. We start by blowing up $|L_{g+k}|$ along $|L_{g-\delta}|$ and denote by $E_{g-\delta}$ the exceptional divisor. We then blow up the strict transform of $|L_{g-\delta+1}|$ and denote by $E_{g-\delta+1}$ the exceptional divisor, and so on until we finally blow up the strict transform of $|L_{g+k-2}|$ and get the last exceptional divisors $E_{g+k-2}$. Let 
$$\pi_{k,\delta}: \widetilde {|L_{g+k}|} \longrightarrow |L_{g+k}|$$
be the composition of these $k+\delta-1$ blow-ups and still write $E_j$ for the strict transforms of the exceptional divisors in $\widetilde {|L_{g+k}|}$.

 We stress that for all $g-\delta\leq j\leq g+k-2$ the fiber of the restriction $\pi_{k,\delta}|_{E_j}:E_j\to |L_j|$ over any $X\in |L_j|$ is  the projective space $\mathbb P(N_{|L_j|/|L_{g+k}|,X})=\mathbb P^{g+k-j-1}$ blown-up at the point $\mathbb P(N_{|L_j|/|L_{j+1}|,X})=\mathbb P^{0}$ (so that the exceptional divisor is $\mathbb P(N_{|L_{j+1}|/|L_{g+k}|,X})=\mathbb P^{g+k-j-2}$) and then at the strict transform of the line $\mathbb P(N_{|L_{j}|/|L_{j+2}|,X})=\mathbb P^{1}$ (with exceptional divisor given by a $\mathbb P(N_{|L_{j+2}|/|L_{g+k}|,X})=\mathbb P^{g+k-j-3}$-bundle) and so on, until finally at the strict transform of $\mathbb P(N_{|L_{j}|/|L_{g+k-2}|,X})$ (the exceptional divisor over it being a $\mathbb P(N_{|L_{g+k-2}|/|L_{g+k}|,X})=\mathbb P^1$-bundle). 

For every $1\leq i\leq k$, we consider the line bundle
$$
\mathfrak H_i:=\pi_{k,\delta}^*\mcO_{|L_{g+k}|}(1)-\sum_{j=i}^{\delta+k-1}E_{g+k-1-j},
$$
whose general section is the strict transform of a general hyperplane in $|L_{g+k}|$ containing $|L_{g+k-1-i}|$. Denoting by $\widetilde H_i$ the strict transform of $H_i$ in $\widetilde {|L_{g+k}|}$, we get that
\begin{eqnarray*}
&\widetilde{H_1}\in |\mathfrak H_1|\simeq \mathbb P^1,\\
&\widetilde{H_i}+E_{g+k-i}\in |\mathfrak H_i|\simeq \mathbb P^i,\textrm{     for }  2\leq i\leq k.
\end{eqnarray*}
\begin{prop}\label{key1}
Fix $g\geq 2$ and $0\leq \delta\leq g-1$, and for every $k\geq 1$ let $\pi_{k,\delta}: \widetilde {|L_{g+k}|} \longrightarrow |L_{g+k}|$ be the composition of $k+\delta-1$ blow-ups described above. Denoting by $\widetilde{V_\delta(S,L_{g+k}) }$ the strict transform of $\overline{V_\delta(S,L_{g+k}) }$ in $\widetilde {|L_{g+k}|}$, the following isomorphism holds in $\widetilde {|L_{g+k}|}$: 
\begin{equation}\label{fondamentale1}
\begin{split}
\widetilde{V_\delta(S,L_{g}) }&\simeq\widetilde{V_\delta(S,L_{g+k}) }\cap \widetilde{H_1}\cap \left(\bigcap_{i=2}^k \widetilde{H_i}+E_{g+k-i} \right)=\\
&=\widetilde{V_\delta(S,L_{g+k}) }\cap \widetilde{H_1}\cap \left(\bigcap_{i=2}^k E_{g+k-i} \right).
\end{split}
\end{equation}
\end{prop}
\begin{proof}
We will prove, by induction on $1\leq l\leq k$, that 
\begin{equation}\label{indu}
\begin{split}
\widetilde{V_\delta(S,L_{g+k-l}) }&\simeq\widetilde{V_\delta(S,L_{g+k}) }\cap \widetilde{H_1}\cap \left(\bigcap_{i=2}^l \widetilde{H_i}+E_{g+k-i} \right)=\\
&=\widetilde{V_\delta(S,L_{g+k}) }\cap \widetilde{H_1}\cap \left(\bigcap_{i=2}^l E_{g+k-i} \right).
\end{split}
\end{equation}
The case $l=1$ amounts to show that:
\begin{equation}\label{piove}\widetilde{V_\delta(S,L_{g+k}) }\cap \widetilde{H_1}\simeq \widetilde{V_\delta(S,L_{g+k-1}) }.
\end{equation}
 Proposition \ref{cz} yields the following equality in $|L_{g+k}|$:
$$
 \overline{V_\delta(S,L_{g+k}) }\cap H_1=\bigcup_{h=0}^\delta  \overline{V_{\delta-h}(S,L_{g+k-1-h}) }.
 $$
In particular, for every $1\leq h\leq \delta$ the intersection $ \overline{V_\delta(S,L_{g+k}) }\cap |L_{g+k-1-h}|$ has codimension $1$ in $ \overline{V_\delta(S,L_{g+k}) }$. For a general $X\in V_{\delta-h}(S,L_{g+k-1-h})$ one has the identification $$N_{V_{\delta-h}(S,L_{g+k-1-h}) / \overline{V_\delta(S,L_{g+k}) },X}=N_{\overline{V_\delta(S,L_{g+k}) }\cap H_1/ \overline{V_\delta(S,L_{g+k}) },X},$$ 
and thus $\widetilde{V_\delta(S,L_{g+k}) }$ intersects the fiber over $X$ of the exceptional divisor $E_{g+k-1-h}\to |L_{g+k-1-h}|$ at the point $$\xi_X=\mathbb P(N_{\overline{V_\delta(S,L_{g+k}) }\cap H_1/ \overline{V_\delta(S,L_{g+k}) },X})\in \mathbb P(N_{|L_{g+k-2}|/|L_{g+k|}}).$$

 On the other hand, the fiber of $\widetilde{H_1}\cap E_{g+k-1-h}\to |L_{g+k-1-h}|$ over $X$ equals $\mathbb P(N_{|L_{g+k-1-h}|/H_1,X})=\mathbb P^{h-1}$ blown-up at the point $\mathbb P(N_{|L_{g+k-1-h}|/|L_{g+k-h}|,X})$ and then, one after the other, at all the strict transforms of the subspaces $\mathbb P(N_{|L_{g+k-1-h}|/|L_{g+k-h+j}|,X})$ for $0\leq j\leq h-2$ (so that the corresponding  exceptional divisors is a $\mathbb P(N_{|L_{g+k-h+j}|/H_1,X})=\mathbb P^{h-j-2}$-bundle). 
For a general choice of $X$, there is a deformation of $X$ inside $\overline{V_\delta(S,L_{g+k})}$ pointing out of $\overline{V_\delta(S,L_{g+k})}\cap H_1$. Hence, $\xi_X$ is not contained in $\widetilde{H_1}\cap E_{g+k-1-h}$ and \eqref{piove} follows.
 
Standard properties of blow-ups (cf. \cite[Prop. IV-21]{EH}) yield $\widetilde H_1\simeq \widetilde{|L_{g+k-1}|}$ and, under this isomorphism, the exceptional divisors $E_{g+k-3},\ldots, E_{g-\delta}$  restricts to the $\delta+k-2$ exceptional divisors of $\pi_{k-1,\delta}:\widetilde{|L_{g+k-1}|}\to |L_{g+k-1}|$. Furthermore, $(\widetilde{H}_2+E_{g+k-2})\cap \widetilde{H_1}=E_{g+k-2}\cap \widetilde{H_1}$ (where the equality follows from the fact that $\widetilde{H_1}\cap\widetilde{H_2}=\emptyset$ by construction) is isomorphic to the strict transform of $|L_{g+k-2}|$ under $\pi_{k-1,\delta}$ and is thus isomorphic to $\widetilde{|L_{g+k-2}|}$. In this way, by induction on  $2\leq l\leq k$, one  obtains the isomorphism
$$
\widetilde{|L_{g+k-l}|}\simeq \widetilde{H_1}\cap \left(\bigcap_{i=2}^{l} \widetilde{H_i}+E_{g+k-i} \right)= \widetilde{H_1}\cap \left(\bigcap_{i=2}^{l} E_{g+k-i} \right)
$$
and shows that, if $l\leq k-1$, the divisor $\widetilde{H_{l+1}}+E_{g+k-l-1}$ (or equivalently, $E_{g+k-l-1}$) restricts to the strict transform of $|L_{g+k-l-1}|$ under $\pi_{k-l,\delta}: \widetilde{|L_{g+k-l}|}\to |L_{g+k-l}|$.  Assuming now that \eqref{indu} holds for $l$, we obtain it for $l+1$ because
\begin{equation*}
 \begin{split}
\widetilde{V_\delta(S,L_{g+k}) }\cap& \widetilde{H_1}\cap \left(\bigcap_{i=2}^{l+1} \widetilde{H_i}+E_{g+k-i} \right)\\
&=\left(\widetilde{V_\delta(S,L_{g+k}) }\cap \widetilde{H_1}\cap \left(\bigcap_{i=2}^{l} \widetilde{H_i}+E_{g+k-i} \right)\right)\cap \left( \widetilde{H_{l+1}}+E_{g+k-l-1} \right)\\
&=\left(\widetilde{V_\delta(S,L_{g+k}) }\cap \widetilde{H_1}\cap \left(\bigcap_{i=2}^{l} E_{g+k-i} \right)\right)\cap  E_{g+k-l-1} \\
&\simeq \widetilde{V_\delta(S,L_{g+k-l}) }\cap \widetilde{|L_{g+k-l-1}|}\simeq  \widetilde{V_\delta(S,L_{g+k-l-1}) }.
\end{split}
\end{equation*}
\end{proof}
We now consider the morphism 
$$\Psi=(\Psi_1,\ldots,\Psi_k):  \widetilde {|L_{g+k}|} \longrightarrow \mathbb P^1\times\cdots\mathbb P^k,$$
whose $i$th component $\Psi_i$ is defined by the linear system $|\mathfrak H_i|$; in particular, $\Psi_i$ is a (non-minimal) resolution of the projection $\pi_i:|L_{g+k}|\dashrightarrow\mathbb P^i$ of $|L_{g+k}|$ from $|L_{g+k-1-i}|$. 

For $2\leq j\leq k$ consider the projective subspace $W_j=\overline {\pi_{k}(|L_{g+k-j}}|\subset\mathbb P^{k}$ of codimension $j$, and blow-up $\mathbb P^k$ first at the point $W_{k}$, then at the strict transform of the line $W_{k-1}$ and so on on, until finally at the strict transform of $W_{2}$. We denote by $\widetilde{\mathbb P^{k}}\to \mathbb P^k$ the composition of these $k-1$ blow-ups and recall that $\mathbb P^k$ is naturally a subvariety of the product $\mathbb P^1\times \cdots\times \mathbb P^k$. We denote by 
\begin{equation}\label{psi}\psi=(\psi_1,\ldots,\psi_k):\widetilde{V_\delta(S,L_{g+k}) }\longrightarrow \mathbb P^1\times\cdots\mathbb P^k\end{equation} the restriction of $\Psi$ to $\widetilde{V_\delta(S,L_{g+k}) }$ and prove the following result.
\begin{lem}\label{blow}
The images of both $\Psi$ and $\psi$ coincide with $\widetilde{\mathbb P^{k}}$. Furthermore, $\widetilde{V_\delta(S,L_{g}) }$ is isomorphic to a fiber of $\psi$. 
\end{lem}
\begin{proof}
We first show, by induction on $h$, that the image of $$(\Psi_1,\ldots,\Psi_h):\widetilde {|L_{g+k}|} \longrightarrow \mathbb P^1\times\cdots\times \mathbb P^h$$ coincides with subvariety $\widetilde{\mathbb P^{h}}$ obtained by blowing-up $\mathbb P^h$ first at the point  $Q_h:=\overline {\pi_{h}(|L_{g+k-h}}|\in\mathbb P^{h}$, then at the strict transform of the line $\overline {\pi_{h}(|L_{g+k-h+1}}|$ and so on, until finally, at the strict transform of the codimension $2$ subspace $\overline {\pi_{h}(|L_{g+k-2}}|\subset\mathbb P^{h}$. The case $h=1$ is trivial and the statement for $h-1$ yields that the image of $(\Psi_1,\ldots,\Psi_h)$ is contained in $\widetilde{\mathbb P^{h-1}}\times \mathbb P^h$. Denoting by $q:\widetilde{\mathbb P^{h}}\to \mathbb P^{h}$ the blow-up map, the universal property of blow-ups yields a factorization $	\Psi_h=\widetilde{\Psi_h}\circ q$ for some morphism $$\widetilde{\Psi_h}:\widetilde {|L_{g+k}|} \to \widetilde{\mathbb P^h}.$$ By resolving the projection $\mathbb P^h\dashrightarrow  \mathbb P^{h-1}$ from the point $Q_h$, we obtain a morphism $\mathrm{Bl}_{Q_h}\mathbb P^h\to \mathbb P^{h-1}$ and an isomorphism $\widetilde{\mathbb P^{h}}\simeq \widetilde{\mathbb P^{h-1}}\times_{\mathbb P^{h-1}}\mathrm{Bl}_{Q_h}\mathbb P^h$. In particular, by considering the projection  $p:\widetilde{\mathbb P^{h}}\to \widetilde{\mathbb P^{h-1}}\subset \mathbb P^1\times\cdots\times \mathbb P^{h-1}$, we get a factorization $(\Psi_1,\ldots,\Psi_{h-1})=\widetilde{\Psi_h}\circ p$. We have thus proved that $$(\Psi_1,\ldots,\Psi_h)=\widetilde{\Psi_h}\circ (p,q)$$ and, since the map $(p,q): \widetilde{\mathbb P^{h}}\to \widetilde{\mathbb P^{h-1}}\times \mathbb P^h\subset \mathbb P^1\times\cdots\times \mathbb P^h$ is the natural inclusion, this implies that the image of $(\Psi_1,\ldots,\Psi_h)$ is $\widetilde{\mathbb P^{h}}$.

We now denote by $e_0,\ldots, e_{k-2}$ the exceptional divisors of $q: \widetilde{\mathbb P^{k}}\to \mathbb P^k$, by numbering them so that $q(e_j)$ has dimension $j$. By construction,  we have $E_{g+j}=\Psi^*(e_{j})$ for every $0\leq j\leq k-2$. Furthermore, on $\widetilde{\mathbb P}^k$ there exist divisors $\widetilde{D_k}\in  |q^*\mathcal{O}_{\mathbb P^k}(1)| $ and $\widetilde {D_i}\in |q^*\mathcal{O}_{\mathbb P^k}(1)-\sum_{j=i-1}^{k-2}e_j|$ for $1\leq i\leq k-1$ such that $\widetilde{H_i}=\Psi^*{\widetilde D_i}$ for $1\leq i\leq k$. It then follows that the intersection  
$$ \widetilde{H_1}\cap \left(\bigcap_{i=2}^k \widetilde{H_i}+E_{g+k-i} \right)
$$
is the inverse image under $\Psi$ of $\widetilde {D_1}\cap \left(\bigcap_{i=2}^k\widetilde {D_i}+ e_{k-i}\right)$ and this consists of the point $\xi\in e_0\cap e_1\cap\ldots\cap e_{k-2}$ determined by $\mathbb P(N_{D_1/\mathbb P^k,Q_k})$, where $D_1=q(\widetilde D_1)$. Proposition \ref{key1} then implies that $\psi^{-1}(\xi)\simeq\widetilde{V_\delta(S,L_{g}) }$; since this has dimension $g-\delta$, the restriction $\psi$ of $\Psi$ to $\widetilde{V_\delta(S,L_{g}) }$ still surjects onto $\widetilde{\mathbb P^{k}}$.
\end{proof}
\begin{thm}\label{main}
For every $g\geq 1$ and every $0\leq \delta\leq g-1$ the closure of the Severi variety $\overline{V_\delta(S,L_{g}) }\subset |L_g|$ is connected. 
\end{thm}
\begin{proof}
We fix $k>>0$ and consider the intersection $|L_g|\cap \overline{V_\delta(S,L_{g+k}) }$ inside $|L_{g+k}|$. Since $\delta$ is small with respect to $g+k$, then $\overline{V_\delta(S,L_{g+k}) }$ is irreducible by Proposition \ref{irreducible}. Since $\overline{V_\delta(S,L_{g}) }=\pi_{k,\delta}(\widetilde{V_\delta(S,L_{g}) }$, it is enough to prove the connectedness of $\widetilde{V_\delta(S,L_{g}) }$, which is a fiber of $\psi:\widetilde{V_\delta(S,L_{g}) }\to \widetilde{\mathbb P^k}$ by Lemma \ref{blow}. 

By construction, the fiber $\pi_{k,\delta}^{-1}(X)$ of $\pi_{k,\delta}$ over any $X\in |L_{g-1}|\setminus  |L_{g-2}|$ is a $\mathbb P^k=\mathbb P(N_{|L_{g-1}|/|L_{g+k}|,X})$ blown-up at the point $\mathbb P(N_{|L_{g-1}|/|L_{g}|,X})$ (the exceptional divisor being identified with $\mathbb P(N_{|L_{g}|/|L_{g+k}|,X})$) and then at the strict transform of the line $\mathbb P(N_{|L_{g-1}|/|L_{g+1}|,X})$ (with exceptional divisor given by a $\mathbb P(N_{|L_{g+1}|/|L_{g+k}|,X})=\mathbb P^{k-2}$-bundle) and so on until, finally, at the strict transform of $\mathbb P(N_{|L_{g-1}|/|L_{g+k-2}|,X})$ (the exceptional divisor over it being a $\mathbb P(N_{|L_{g+k-2}|/|L_{g+k}|,X})=\mathbb P^1$-bundle). Hence, $\Psi$ maps $\pi_{k,\delta}^{-1}(X)$ isomorphically onto $\widetilde{\mathbb P^k}$, that is, $\pi_{k,\delta}^{-1}(X)$ defines a section of $\Psi: \widetilde{|L_{g+k}|}\to \widetilde{\mathbb P^k}$. In order to ensure that this restricts to a section of $\psi$,  we need to choose $X$ so that $X=(k+1)J+ C\in |L_{g-1}|\subset |L_{g+k}|$ for some irreducible $ C\in |L_{g-1}|$ with precisely $\delta$ nodes, that is, $ C\in V_\delta(S,L_{g-1})$. We claim that with this choice of $X$ the fiber $\pi_{k,\delta}^{-1}(X)$ is contained in $\widetilde{V_\delta(S,L_{g+k})}$ and thus defines a section of $\psi$. Assuming this, we may apply, e.g., \cite[Lem. 3]{KoL} and conclude that a general fiber of $\psi$ is connected. Since $\widetilde{\mathbb P^k}$ is smooth, by a classical argument using the Stein factorization of $\psi$ (cf., e.g., \cite[proof of Thm. 2.1]{FL}) we obtain connectedness of any fiber of $\psi$ and thus in particular of $\widetilde{V_\delta(S,L_{g}) }$.

It remains to prove the claim. Proposition \ref{cz} yields
$$
V_\delta(S,L_{g-1})\subset \overline{V_{\delta+k}(S,L_{g+k})}\subset \overline{V_{\delta}(S,L_{g+k})},
$$
and we recall that $\overline{V_{\delta}(S,L_{g+k})}$ is singular at the points of $\overline{V_{\delta+k}(S,L_{g+k})}$ (cf. \cite{DH}). More precisely, in a neighborhood of a general $Y\in V_{\delta+k}(S,L_{g+k})$ the locus $\overline{V_{\delta}(S,L_{g+k})}$ is the union of $\delta+k \choose \delta$ sheets corresponding to the possibilities of choosing $\delta$ nodes among the $\delta+k$ nodes of $Y$. Let $V$ be the sheet of $\overline{V_{\delta}(S,L_{g+k})}$ around the point $X=(k+1)J+ C$ corresponding to the choice of the $\delta$ nodes of $X$ lying on $C$, and denote by $N$ the scheme of nodes of $C$. In order to show that $\pi_{k,\delta}^{-1}(X)$ is contained in $\widetilde{V_\delta(S,L_{g+k})}$, we need to check that 
\begin{equation}\label{tennis1}
N_{V\cap |L_{g-1+h}|/V,X}\simeq N_{|L_{g-1+h}|/|L_{g+k}|,X},\,\, \textrm{  for every  }0\leq h\leq  k-1.\end{equation}
Since $V_{\delta}(S,L_{g-1})$ is smooth at $C$, then $H^1(S,L_{g-1}\otimes I_N)=0$ and one obtains $H^1(S,L_{g-1+h}\otimes I_N)=0$ for every $1\leq h\leq  k-1$ by considering the following short exact sequence:
$$
0\longrightarrow L_{g-1}\otimes I_N\longrightarrow L_{g-1+h}\otimes I_N\longrightarrow L_{g-1+h}|_{hJ}\longrightarrow 0.
$$
We thus have the following commutative diagram:
$$
\xymatrix{
&0\ar[d]&0\ar[d]\\
0\ar[r]&H^0(S,L_{g-1+h}\otimes I_N)\ar[r]\ar[d]&H^0(S,L_{g+k}\otimes I_N)\ar[r]\ar[d]&H^0(L_{g+k}|_{(k+1-h)J})\ar[r]\ar@{=}[d]&0\\
0\ar[r]&H^0(S,L_{g-1+h})\ar[r]\ar[d]&H^0(S,L_{g+k})\ar[r]\ar[d]&H^0(L_{g+k}|_{(k+1-h)J})\ar[r]&0.\\
&H^0(\mathcal O_N)\ar[d]\ar@{=}[r]&H^0(\mathcal O_N)\ar[d]\\
&0&0\\
}
$$
The middle horizontal short exact provides an identification of the normal space $N_{|L_{g-1+h}|/|L_{g+k}|,X}$ with $H^0(L_{g+k}|_{(k+1-h)J})\simeq \mathbb C^{k+1-h}$. Since the sheet $V$ corresponds to the choice of the $\delta$ nodes of $X$ lying on $C$, the projective tangent space to $V$ at $X$ is isomorphic to $\mathbb{P}(H^0(S,L_{g+k}\otimes I_N))$, and  analogously the projective tangent space to $V\cap |L_{g-1+h}|$ at $X$ is isomorphic to $\mathbb{P}(H^0(S,L_{g-1+h}\otimes I_N))$. Hence, the upper horizontal short exact sequence identifies the normal space $N_{V\cap |L_{g-1+h}|/V,X}$ with $H^0(L_{g+k}|_{(k+1-h)J})$ and this implies the desired isomorphisms \eqref{tennis1}.
\end{proof}
\subsection{Connectedness result for expanded degenerations}
We consider the moduli stack $ \M_{g-\delta}(S/J,L_g)$ of stable relative maps to expanded degenerations of $(S,J)$ with multiplicity $1$ along  the relative divisor $J$, introduced by Jun Li \cite{Li1,Li2}. An expanded degeneration of $S$ along $J$ is a semistable model of $S$ 
$$
S[n]_0:=S\cup_JR\cup_J\cdots\cup_JR,\,\,\,\,R:=\PP(\mathcal O_J\oplus  N_{J/S})
$$ which is union of $S$ with a length-$n$ tree of ruled surfaces $R$ as above for some $n\geq 0$. More precisely, denoting by $J_0$ and $J_{\infty}$ the two distinguished sections on $R$ such that $N_{J_0/R}\simeq N_{J/S}^\vee$ and $N_{J_\infty/R}\simeq N_{J/S}$, the above expansion $S[n]_0$ is obtained by gluing the first copy of $R$ with $S$ along $J_0$, while two adjacent copies of $R$ are glued identifying the $J_{\infty}$ on the left surface with the $J_0$ on the right one. The relative divisor $J$ is the section $J_\infty$ on the latter copy of $R$. 

Since $J\cdot L_g=1$ on $S$, stable maps on expansions of $S$ can be easily described. In particular, a stable relative map  of genus $g-\delta$ to the expansion $S[n]_0$ is a map $f:C\to S[n]_0$ from a connected prestable curve of genus $g-\delta$ such that no component of $C$ is mapped entirely to the singular locus of $S[n]_0$ or to the divisor $J_{\infty}$ on the last copy of $R$, the inverse image of every component of the singular locus of $S[n]_0$ is a node of $C$ connecting two irreducible components of $C$ that are mapped to two adjacent components of $S[n]_0$, and the automorphism group of $f$ is finite. Furthermore, in order for $f$ to define a point of $\M_{g-\delta}(S/J,L_g)$, we require that the image of the curve $f(C)\subset S[n]_0$ under the projection $S[n]_0\to S$ lies in $|L_g|$. Any such map can be thus decomposed as 
\begin{equation}\label{stabledec}
f=f_0\cup\cdots\cup f_n:C=C_0\cup C_1\cup \cdots \cup C_n\to S[n]_0,
\end{equation} where every $C_i$ is an irreducible curve, two adjacent $C_i$ share a node of $C$, and $f_i(C_i)$  is contained in the $i$-th copy of $R$ if $i\geq 1$, while $f_0(C_0)\subset S$. Denoting by $h_i\geq 0$ the arithmetic genus of $C_i$ and by $g_i\geq h_i$ the arithmetic genus of $f(C_i)$, the integers $h_i$ and $g_i$ satisfy $\sum_{i=0}^ng_i=g$ and $\sum_{i=0}^nh_i=g-\delta$. Furthermore, the curve $f_0(C_o)\in |L_{g_0}|$, while for $i\geq 1$ the numerical class of $f_i(C_i)$ is $g_iJ_0+f$ on $R$ (where $f$ is the class of a fiber of $R\to J$), and its linearly equivalence class is determined by the gluing condition as follows. Setting $\underline g=(g_0,\ldots,g_n)$ and $x_1(\underline g):=p_{10}(g_0+g_1)$, we have $f(C_1)\in |N_{1}(\underline g)|$ with $N_{1}(\underline g):= g_1J_0+f_{x_1(\underline g)}$. As in \cite[\S 2]{FT}, one verifies that $N_{1}(\underline g)$ has two base points, namely, $p_{10}(g_0)\in J_0$ and $x_1(\underline g)\in J_{\infty}$. Analogously, setting $x_i(\underline g):=p_{10}(g_0+\cdots+g_{i})$ for every $1\leq i\leq n$, we have $f(C_i)\in |N_{i}(\underline g)|$ where the line bundle $N_i(\underline g):=g_iJ_0+f_{x_i(\underline g)}$ has base points $x_{i-1}(\underline g)\in J_0$ and $x_{i}(\underline g)\in J_\infty$. In particular, the evaluation map
$$\M_{g-\delta}(S/J,L_g)\to J$$
at the relative divisor always takes the value $p_{10}(g)$. The stability condition implies that a stable map to $S[n]_0$ has no component mapping to a fiber of a ruled surface $R$. In particular, we have $g_i>0$ for $i>0$ and, since the linear systems $|N_i(\underline g)|$ contains no rational curve when $g_i>0$, also $h_i>0$.

The multiplicative group $\mathbb C^*$ acts fiberwise on $R$ preserving the sections $J_0$ and $J_\infty$; this induces an action of $(\mathbb C^*)^n$ on $S[n]_0$ for every $n\geq 1$. Two stable maps with target $S[n]_0$ are considered equivalent if they differ by the action of an element of $(\mathbb C^*)^n$ on $S[n]_0$. 
Summing up, thanks to the decomposition \eqref{stabledec}, a point $[f]\in \M_{g-\delta}(S/J,L_g)$ representing a stable map with target $S[n]_0$ defines points of the following moduli stacks
$$[f_0]\in \M_{h_0}(S,L_{g_0})\,\,,\,\,[f_i]\in  \M_{h_i}(R,N_{i}(\underline g))/\mathbb{C}^* \textrm{  for  }i\geq 1.$$
We will later use the following result concerning equigeneric loci in the linear systems $|N_i(\underline g)|$ on $R$.
\begin{prop}\label{ruled}
Let $R$ and $N_i(\underline g):=g_iJ_0+f_{x_i(\underline g)}\in \Pic(R)$ be as above. Then for every integer $1\leq h_i \leq g_i$ the following hold:
\begin{itemize}
\item[(i)] Both the equigeneric locus $V^{h_i}(R,N_i(\underline g))\subset |N_i(\underline g)|$ and the moduli stack of smoothable stable maps $\M_{h_i}(R,N_i(\underline g))^{\mathrm{sm}}$ with image in the linear system $|N_i(\underline g)|$ have pure dimension $ h_i$ and are generically reduced.
\item[(ii)] Let $V$ and $W$ be two intersecting components of $\overline{V^{h_i}(R,N_i(\underline g))}$ and let $Z$ be an irreducible component of $V\cap W$ whose general point parametrizes a curve containing neither $J_0$ nor $J_{\infty}$; then $Z$ has pure codimension $1$ and is generically reduced. 
\end{itemize}
\end{prop}
\begin{proof}
Let $\eta\in \Pic^0(J)$ be the line bundle such that $R=\mathbb P(\mathcal O_J\oplus\eta)$ and denote by $\phi: R\to J$ the natural projection; we have $J_{\infty}\equiv J_0-\phi^*(\eta)$. We recall from \cite{FT} that curves in $|N_i(\underline g)|$ have arithmetic genus $g_i$ and $\dim |N_i(\underline g)|=g_i$. According to our notation, the moduli stack $\M_{h_i}(R,N_{i}(\underline g))$ parametrizes maps $f$ such that $f(C)$ lies in the linear system $|N_{i}(\underline g)|$; this is a closed  substack of the moduli stack $\M_{h_i}(R,g_iJ_0+f)$ where only the numerical class of $f(C)$ is fixed. Let $\M_{h_i}(R,N_{i}(\underline g))^{sm}$ and $\M_{h_i}(R,g_iJ_0+f)^{sm}$ be the closed substacks parametrizing smoothable maps. The deformations of a map $[f]\in \M_{h_i}(R,g_iJ_0+f)^{sm}$ are governed by the normal sheaf $N_f$. As in the proof of Proposition \ref{dimension}(ii)-(iii), one shows that a general $[f]$ in any irreducible component of $\M_{h_i}(R,g_iJ_0+f)^{sm}$ is unramified and thus $\M_{h_i}(R,g_iJ_0+f)^{sm}$ is generically reduced and has pure dimension $h_i+1$. Being a fiber of the evaluation map $\M_{h_i}(R,g_iJ_0+f)^{sm}\to J_0$, the stack $\M_{h_i}(R,N_{i}(\underline g))^{sm}$ is generically reduced and of pure dimension $h_i$. The same holds true for the equigeneric locus $\overline{V^{h_i}(R,N_i(\underline g))}$ thanks to the existence of a birational map $\tilde\mu: \tilde\M_{h_i}(R,g_iJ_0+f)^{sm}\to \overline{V^h(R,N_i(\underline g))}$ from the semi-normalization $ \tilde\M_{h_i}(R,g_iJ_0+f)^{sm}$ of $ \tilde\M_{h_i}(R,g_iJ_0+f)^{sm}$. This proves (i).

To obtain (ii) one proceeds exactly as in the proofs of Proposition \ref{cohen} and Lemma \ref{cz}, where the density of the Severi variety in the equigeneric locus (whose validity on $R$ is unknown) was never used. The proofs work the same way because all curves in the linear system $|N_i(\underline g)|$ are irreducible except for those lying in the two hyperplanes of $|N_i(\underline g)|$ defined by the linear subsystem $J_0+ |(g_i-1)J_0+f_{x_i(\underline g)}|$ and $J_\infty+ |(g_i-1)J_0+f_{y_i(\underline g)}|$, where we have set $y_i(\underline g):= p_{10}(g_0+\cdots+g_{i}-1)$.
\end{proof}
\color{black}

Coming back to $ \M_{g-\delta}(S/J,L_g)$, this is a proper and separated Deligne-Mumford stack by \cite[Thm. 3.10]{Li1}. We briefly recall why it is DM. By Li's construction, there is a scheme $$S[n]\subset S\times \mathbb A^n$$ combining all possible expansions $S[k]_0$ for $0\leq k\leq n$; in particular, a general fiber of the projection $S[n]\to \mathbb A^n$ is isomorphic to $S$, the central fiber over $0\in \mathbb A^n$ is the $n$-th expansion $S[n]_0$, while the fibers over any coordinate ($n-k$)-dimensional plane in $\mathbb A^n$ are isomorphic to $S[k]_0$. The natural action of $(\mathbb C^*)^n$ on $\mathbb A^n$ lifts to an action on $S[n]$ so that its restriction to $S[n]_0$ is trivial on $S$, while the $i$-th copy of $\mathbb C^*$ acts on the $i$-th copy of $R$ fiberwise so that $J_0$ and $J_\infty$ are fixed. Let us consider the projection
$$\beta_n: S[n]\to S$$
and denote by $\M_{g-\delta}(S[n],L_g)$ the DM stack of ordinary stable maps $f:C\to S[n]$ such that the image of $\beta_n(f(C))\in |L_g|$. Let $ \M_{g-\delta}(S[n],L_g)^{Li}$ be the closed DM substack parametrizing stable (in J. Li's sense recalled above) maps to some expanded degeneration $S[k]_0$ with $k\leq n$; this admits a $(\mathbb C^*)^n$ action that is induced by the one on $S[n]$ and has finite stabilizers. The fact that $ \M_{g-\delta}(S/J,L_g)$ is DM thus follows from the existence of a surjective étale map $$\M_{g-\delta}(S[g],L_g)^{Li}/(\mathbb C^*)^g\to  \M_{g-\delta}(S/J,L_g).$$

We denote by $\M_{g-\delta}(S/J,L_g)^{sm}$ the locus in $\M_{g-\delta}(S/J,L_g)$ parametrizing stable maps which are smoothable. In particular, a general point in any component of $ \M_{g-\delta}(S/J,L_g)^{sm}$ parametrizes a stable map with irreducible domain and target precisely $S$. 

Now we fix $k>>0$ as in the proof of Theorem \ref{main} so that $\overline{V_\delta(S,L_{g+k}) }$ is irreducible. In particular, the moduli stack $\M_{g+k-\delta}(S/J,L_{g+k})^{sm}$ and its semi-normalization $\tilde\M_{g+k-\delta}(S/J,L_{g+k})^{sm}$ are irreducible, too. We consider the natural map 
$$
\alpha:\tilde\M_{g+k-\delta}(S/J,L_{g+k})^{sm}\to |L_{g+k}|
$$ 
that sends the class of a stable map $f:C\to S[n]_0$ for $0\leq n\leq g+k$ to the image of $f(C)$ under the map $S[n]_0\to S$. 
\begin{lem}\label{federico}
The map $\alpha$ factors through a map 
\begin{equation}\label{secondo}
\tilde \alpha: \tilde\M_{g+k-\delta}(S/J,L_{g+k})^{sm}\to \widetilde{|L_{g+k}|},\end{equation}
 so that the image of $
\tilde \alpha$ is the strict transform $\widetilde{V_\delta(S,L_{g+k}) }$ of $\overline{V_\delta(S,L_{g+k}) }$ in the blow-up $\widetilde{|L_{g+k}|}$ of $|L_{g+k}|$. 
\end{lem}
\begin{proof}
We exploit the family $S[1]\to \mathbb A^1$, whose general fibers $S_t$ are isomorphic to $S$ and whose central fiber is $S[1]_0=S\cup R$. Denote by $\mathcal L$ the equivalence class (up to twisting by multiples of a component of the central fiber) of line bundles on $S[1]$ restricting to $L_{g+k}$ on a general fiber $S_t$ and to a line bundle of the form $(L_{g_0}, (g+k-g_0)J_0+f_{p_{10}(g+k)})$ with $0\leq g_0\leq g+k-1$ on $S[1]_0$; since all pairs $(L_{g_0}, (g+k-g_0)J_0+f_{p_{10}(g+k)})$ are limits of the same family of line bundles on $S_t$, they are considered equivalent and we denote by $L:=[(L_{g_0}, (g+k-g_0)J_0+f_{p_{10}(g+k)})]$ their class. Let $$\M_{g+k-\delta}(S[1]^\mathrm{exp},\mathcal L)\to \mathbb A^1$$ be the moduli stack of connected stable maps to expanded degenerations of $\chi$ constructed in \cite{Li1,Li2}. Over a point $t\in \mathbb A^1\setminus\{ 0\}$, the fiber of $\chi_\delta$ is simply the moduli stack $\M_{g+k-\delta}(S,L_{g+k})$ of ordinary stable maps on $S$, while the fiber over $0$ is the stack $\M_{g-\delta}(S[1]_0^\mathrm{exp},L)$ parametrizing stable maps (in the sense of Jun Li) to some expanded degeneration of $S[1]_0$, or equivalently, to some expanded degenerations $S[n]_0$ of $S$ with $n\geq 1$. By \cite[\S 3]{Li1}, $\M_{g-\delta}(S[1]_0^\mathrm{exp},L)$  admits the following decomposition:
\begin{equation}\label{arbitro}
\M_{g-\delta}(S[1]_0^\mathrm{exp},L)=\bigcup_{(g_0,h_0)\in I}\M_{h_0}(S/J,L_{g_0})\times  \M_{g+k-\delta-h_0}(R/J_0,(g+k-g_0)J_0+f_{p_{10}(g+k)}),
\end{equation}
where $I$ is the set of indices $$I:=\{(g_0,h_0)\in\mathbb Z^2|\,\, 0\leq g_0< g+k,\,\max\{0,g_0-\delta\}\leq h_0\leq \min\{g_0,g+k-\delta-1\}\},$$
and $\M_{h_0}(S/J,L_{g_0})$ and $\M_{g+k-\delta-h_0}(R/J_0,(g+k-g_0)J_0+f_{p_{10}(g+k)})$ are stacks of stable relative maps to expanded degenerations of $(S,J)$ and $(R,J_0)$, respectively. Each factor in the decomposition \eqref{arbitro} appears with multiplicity $1$ by \cite[Prop. 4.13]{Li1} and,  by \cite[\S 3.1]{Li2}, defines a Cartier divisor in $\M_{g+k-\delta}(S[1]^\mathrm{exp},\mathcal L)$ that we denote by $\mathcal{D}_{h_0,g_0}'$. 

Recalling that points of $\M_{g+k-\delta}(S/J,L_{g+k})$ parametrize maps to expanded degenerations $S[n]_0$  of $S$ with $0\leq n\leq g+k$, there is a natural map
$$\gamma:\M_{g+k-\delta}(S/J,L_{g+k})\to \M_{g+k-\delta}(S[1]^\mathrm{exp},\mathcal L)$$
induced by the projection $S[g+k]\to S[1]$. Pulling back $\mathcal{D}_{h_0,g_0}'$ via $\gamma$, we obtain  a Cartier divisor on $\M_{g+k-\delta}(S/J,L_{g+k})$ that we denote by $\mathcal{D}_{h_0,g_0}$. In particular, for any fixed $0\leq g'< g+k$ the union
\begin{equation}\label{venduto}
\bigcup_{\substack{0\leq g_0\leq g'\\\max\{0,g_0-\delta\}\leq h_0\leq \min\{g_0,g+k-\delta-1\}}}\mathcal{D}_{h_0,g_0}
\end{equation} is a Cartier divisor on $\M_{g+k-\delta}(S/J,L_{g+k})$. On the other hand, $\alpha^{-1}(|L_{g'}|)$ coincides with the pullback under the semi-normalization map 
\begin{equation}\label{semi}\nu: \tilde\M_{g+k-\delta}(S/J,L_{g+k})^{sm}\to \M_{g+k-\delta}(S/J,L_{g+k})^{sm}\end{equation} (which is a universal homeomorphism) of the restriction of \eqref{venduto} to the smoothable locus; this implies that $\alpha^{-1}(|L_{g'}|)$ is a Cartier divisor on $\tilde\M_{g+k-\delta}(S/J,L_{g+k})^{sm}$. The case $g'=g-\delta$ yields, by the universal property of blow-ups \cite[70.17]{St}, a factorization of $\alpha$ through  the blow-up of $|L_{g+k}|$ along $|L_{g-\delta}|$:
\begin{equation*}
\xymatrix{
\tilde\M_{g+k-\delta}(S/J,L_{g+k})^{sm} \ar@/_2pc/[rr]_{\alpha}\ar[r]^-{\alpha_1}& Bl_{|L_{g-\delta}|}|L_{g+k}|\ar[r]^-{b_1}& |L_{g+k}|.}
\end{equation*}
Denoting by $|L_{g-\delta+1}|^t$ the strict transform of $|L_{g-\delta+1}|$ in $Bl_{|L_{g-\delta}|}|L_{g+k}|$, we have
$$
\alpha^{-1}(|L_{g-\delta+1}|)=\alpha_1^{-1}(|L_{g-\delta+1}|^t+E_{g-\delta})=\alpha_1^{-1}(|L_{g-\delta+1}|^t)+\alpha^{-1}(|L_{g-\delta}|),
$$
and thus conclude that $\alpha_1^{-1}(|L_{g-\delta+1}|^t)$ is a Cartier divisor as it is difference of two Cartier divisors. Therefore, $\alpha_1$ factors through the blow-up of $Bl_{|L_{g-\delta}|}|L_{g+k}|$ along $|L_{g-\delta+1}|^t$. By the same argument, after $g+\delta-1$ steps one obtains a factorization of $\alpha$ through a map 
$$
\tilde \alpha: \tilde\M_{g+k-\delta}(S/J,L_{g+k})^{sm}\to \widetilde{|L_{g+k}|}.$$
Since $\tilde\M_{g+k-\delta}(S/J,L_{g+k})^{sm}$ is irreducible and its image under $\alpha$ is $\overline{V_\delta(S,L_{g+k}) }$, the image of $
\tilde \alpha$ is precisely the strict transform $\widetilde{V_\delta(S,L_{g+k}) }$. 
\end{proof}

\begin{thm}\label{expanded}
For every $g\geq 1$ and every $0\leq \delta\leq g-1$ there exists a closed substack $\M_{g-\delta}(S/J,L_g)^\dagger$ of $ \M_{g-\delta}(S/J,L_g)$ that is connected and contains $ \M_{g-\delta}(S/J,L_g)^{sm}$. The image in $S$ of a stable map parametrized by a point $[f]\in\M_{g-\delta}(S/J,L_g)^\dagger$ always lies in the closure of the Severi variety $\overline{V_{\delta}(S,L_g)}$. If moreover $[f]\in\M_{g-\delta}(S/J,L_g)^\dagger$ and its image is a nodal curve, then $[f]\in\M_{g-\delta}(S/J,L_g)^{sm}$. 
\end{thm}
\begin{proof}
We use the same notation as in the proof of Lemma \ref{federico}. We compose $\tilde\alpha$ with the morphism $\psi:\widetilde{V_\delta(S,L_{g+k}) }\longrightarrow \widetilde{\mathbb P^k}$ introduced in \eqref{psi}. As proved in Lemma \ref{blow} and Theorem \ref{main}, $\widetilde{V_\delta(S,L_{g}) }$ is isomorphic to the intersection $$\widetilde{V_\delta(S,L_{g+k}) }\cap \widetilde H_1\cap \left(\bigcap_{i=2}^k \widetilde{H_i}+E_{g+k-i} \right)\simeq \widetilde{V_\delta(S,L_{g+k}) }\cap \widetilde H_1\cap \left(\bigcap_{i=2}^k E_{g+k-i} \right)$$ which is a fiber of $\psi$ and is connected because $\psi$ admits a section and $\widetilde{\PP^k}$ is smooth. 

We will now prove the existence of an injective morphism
\begin{equation}\label{mito}
r:\tilde\alpha^{-1}\left( \widetilde H_1\cap \left(\bigcap_{i=2}^k E_{g+k-i} \right)\right) \to  \M_{g-\delta}(S/J,L_{g})
\end{equation} 
whose image contains $\M_{g-\delta}(S/J,L_{g})^{sm}$ and will be denoted by $\M_{g-\delta}(S/J,L_{g})^{\dagger}$.
To do that, we denote by $\mathcal{D}_{h_0,g_0}^{sm}$ the pullback under the map $\nu$ in \eqref{semi} of the restriction of $\mathcal{D}_{h_0,g_0}$ to the smoothable locus, and observe that 
\begin{equation*}
\tilde\alpha^{-1}(E_{g-\delta})=\alpha^{-1}|L_{g-\delta}|=\bigcup_{\substack{0\leq g_0\leq g-\delta\\\max\{0,g_0-\delta\}\leq h_0\leq \min\{g_0,g+k-\delta-1\}}}\ \mathcal{D}_{h_0,g_0}^{sm},
\end{equation*} 
and then 
\begin{equation*}
\begin{split}
\tilde\alpha^{-1}(E_{g-\delta+1})=\alpha^{-1}&|L_{g-\delta+1}|-\tilde\alpha^{-1}(E_{g-\delta})=\\&=\bigcup_{\max\{0,g-2\delta+1\}\leq h_0\leq \min\{g-\delta+1,g+k-\delta-1\}} \mathcal{D}_{h_0,g-\delta+1}^{sm}.
\end{split}
\end{equation*} 

Analogously, for every $1\leq i\leq k+\delta-2$ one has
\begin{equation*}
\begin{split}
\tilde\alpha^{-1}(E_{g-\delta+i})=\alpha^{-1}&|L_{g-\delta+i}|-\sum_{j=g-\delta}^{g-\delta+i-1}\tilde\alpha^{-1}(E_{j})=\\&=\bigcup_{\max\{0,g-2\delta+i\}\leq h_0\leq \min\{g-\delta+i,g+k-\delta-1\}} \mathcal{D}_{h_0,g-\delta+i}^{sm}.
\end{split}
\end{equation*} 
Hence, a general point of any irreducible components of $\tilde\alpha^{-1}(E_{g-\delta+i})$ represents a stable morphism $f=f_0\cup f_1:C_0\cup C_1\to S[1]_0$ such that $[f_0]\in \M_{h_0}(S,L_{g-\delta+i})$ and $[f_1]\in \M_{g+k-\delta-h_0}(R,(k+\delta-i)J_0+f_{p_{10}(g+k)})$ for some $\max\{0,g-2\delta+i\}\leq h_0\leq \min\{g-\delta+i,g+k-\delta-1\}$. 
Finally, one computes that 
$$\tilde\alpha^{-1}(\widetilde H_1)=\alpha^{-1}|L_{g+k-1}|-\sum_{j=g-\delta}^{g+k-2}\tilde\alpha^{-1}(E_{j})\simeq \mathcal{D}_{g+k-\delta-1,g+k-1},
$$
so that a general point of any irreducible components of $\tilde\alpha^{-1}(\widetilde H_1)$ represents a stable map $$f=f_0\cup f_1:C_0\cup J\to S[1]_0$$ with $[f_0]\in \M_{g+k-\delta-1}(S,L_{g+k-1})$ and $[f_1]\in \M_1(R,J_0+f_{p_{10}(g+k)})$. A general point in the intersection $\tilde\alpha^{-1}(\widetilde H_1\cap E_{g+k-2})$ then represents a map $f$ with a chain of at least $2$ copies of $J$, that is, $f$ can be written as $$f=f_0\cup f_1\cup f_2:C_0\cup J\cup J \to S\cup R\cup R=S[2]_0$$ with $f(C_0)\in |L_{g+k-2}|$, $f(C_1)\in |J_0+f_{p_{10}(g+k-1)}|$ and $f(C_2)\in |J_0+f_{p_{10}(g+k)}|$. By further intersecting with $\tilde\alpha^{-1}(E_{g+k-3})$, we select maps with a chain of at least $3$ copies of $J$, and so on. 

In conclusion, every point of $\tilde\M_{g+k-\delta}(S/J,L_{g+k})^{sm}$ lying in the intersection $\tilde\alpha^{-1}\left( \widetilde H_1\cap \left(\bigcap_{i=2}^k E_{g+k-i} \right)  \right)$ parametrizes a stable map $f$ with a chain of $k$ copies of $J$, that is, $f$ can be written as $$f= f_0\cup f_1\cup\ldots\cup f_k:C_0\cup J\cup\ldots J\to S[m]_0\cup R\cup\ldots\cup R=S[m+k]_0$$ for some $0\leq m\leq g$ so that $[f_0]\in \M_{g-\delta}(S/J,L_g)$ and $[f_i]\in \M_{1}(R,J_0+f_{p_{10}(g+i)})/\mathbb C^*$ for $1\leq i\leq k$. In particular, for $1\leq i\leq k$ the class of $f_i$ is uniquely determined and this proves the existence of an injective morphism $r$  as in \eqref{mito}. Its image, denoted by 
$ \M_{g-\delta}(S/J,L_g)^{\dagger}$, is a closed substack of $ \M_{g-\delta}(S/J,L_g)$. Let $f_0:C_0\to S$ define a general point of $ \M_{g-\delta}(S,L_g)^{sm}$ and take $[f_i]\in \M_{1}(R,J_0+f_{p_{10}(g+i)})/\mathbb C^*$ for $1\leq i\leq k$; the map $f= f_0\cup f_1\cup\ldots\cup f_k$  is unramified and thus, by dimensional reasons using Propositions \ref{dimension} and \ref{ruled}, one may check that $[f]\in\tilde\M_{g+k-\delta}(S/J,L_{g+k})^{sm}$. This proves that $ \M_{g-\delta}(S/J,L_g)^{\dagger}$ contains $ \M_{g-\delta}(S/J,L_g)^{sm}$.

By construction the stack $ \M_{g-\delta}(S/J,L_g)^{\dagger}$ is homeomorphic to a fiber of $\psi \circ\tilde \alpha$ and thus, in order to prove that it is connected, it only remains to show that the section of $\psi$ constructed in the proof of Theorem \ref{main} lifts to a section of $\psi \circ\tilde \alpha$. 

Take $X=(k+1)J+ C\in |L_{g-1}|\subset |L_{g+k}|$ with $ C\in V_\delta(S,L_{g-1})$ as in the proof of Theorem \ref{main}. By dimensional reasons using Propositions \ref{dimension} and \ref{ruled}, all stable maps $f=f_0\cup f_1$ with $f_0:\widetilde C\to C\subset S$ the normalization of $C$ and $[f_1]\in  \M_{k+1}(R/J_0,(k+1)J_0+f_{p_{10}(g)})^{sm}/\mathbb C^*$ are smoothable and thus contained in the fiber $\alpha^{-1}(X)$. Setting $\M:=\M_{k+1}(R/J_0,(k+1)J_0+f_{p_{10}(g)})^{sm}$, this provides an inclusion $$\M/\mathbb C^*\subset\tilde\alpha^{-1}(\pi_{k,\delta}^{-1}(X))=\alpha^{-1}(X).$$ The open substack $\M^\circ\subset \M$ parametrizing maps with irreducible domain is isomorphic to the open subset $U\subset|(k+1)J_0+f_{p_{10}(g)}|\simeq \PP^{k+1}$ parametrizing irreducible curves. It is trivial to check that $U$ is the complement of two hyperplanes and the restriction to $\M^\circ/\mathbb C^* \simeq U/\mathbb C^*$ of $$\tilde \alpha|_{\M/\mathbb C^*} : \M/\mathbb C^* \to \pi_{k,\delta}^{-1}(X)\simeq \widetilde{\PP^k}$$ is an isomorphism onto its image. In particular, $\tilde \alpha|_{\M/\mathbb C^*}$ is birational and thus an isomorphism by Zariski's Main Theorem. Since $\pi_{k,\delta}^{-1}(X)$ is a section of $\psi$, we get that $\M/\mathbb C^* $ is a section of $\psi\circ \tilde\alpha$ and we can conclude that all fibers of $\psi\circ \tilde\alpha$ are connected by considering the Stein factorization of $\psi\circ \tilde\alpha$ and again applying Zariski's Main Theorem.

As concerns the last part of the statement, let $$f= f_0\cup f_1\cup\ldots\cup f_n:C_0\cup C_1\cup\ldots C_n\to S[n]_0$$ define a point of $ \M_{g-\delta}(S/J,L_g)^{\dagger}\setminus  \M_{g-\delta}(S/J,L_g)^{sm}$. Again by dimensional reasons using Propositions \ref{dimension} and \ref{ruled}, for some $i$ the map $f_i$ contracts a component $C_i'$ of $C_i$. Since $ \M_{g-\delta}(S/J,L_g)^{\dagger}\subset \M_{g+k-\delta}(S/J,L_{g+k})^{sm}$, the image curve $f(C_i)$ has a singularity at the point $f_i(C_i')$ that is worse than a node by, e.g., \cite{Va}. 
\end{proof}
We will now define a moduli stack $ \V_{\delta}(S/J,L_g)$ of stable relative $\delta$-nodal curves to expanded degenerations of $(S,J)$ with multiplicity $1$ along  the relative divisor $J$. We apply Li and Wu's construction \cite{LW} (cf. also \cite{Li3} for a nice survey) of stacks of relative ideal sheaves with fixed Hilbert polynomial to obtain a moduli stack $|L_g|^\mathrm{exp}$, whose closed points parametrize equivalence classes of curves $X=X_0\cup X_1\cup\ldots\cup X_n$ with finite automorphism group living in some expanded degeneration $S[n]_0$ of $S$, such that $X$ has no component contained in the singular locus of $S[n]_0$ or in the divisor $J_\infty$ on the $n$-th copy of $R$, and $X_0\in |L_{g_0}|$ while $X_i\in |N_i(\underline g)|$ for $1\leq i\leq n$. Two such curves $X$ and $X'$ define the same point of $|L_g|^\mathrm{exp}$ if they live in the same $S[n]_0$ and lie in the same orbit under the natural action of $(\CC^*)^n$. The stack $|L_g|^\mathrm{exp}$ is Deligne-Mumford, proper, separated and of finite type. 

For every $0\leq \delta\leq g$, we define $\overline{\V_{\delta}(S/J, L_g)}$ to be the closure in $ |L_g|^\mathrm{exp}$ of the Severi variety $V_{\delta}(S, L_g)$. Denoting by $\tilde\M_{g-\delta}(S/J,L_g)^{\mathrm sm}$ the semi-normalization of $\M_{g-\delta}(S/J,L_g)^{\mathrm sm}$, the stack $\overline{\V_{\delta}(S/J, L_g)}$ can be alternatively described as the image of the natural map $$\tilde \mu:\tilde\M_{g-\delta}(S/J,L_g)^{\mathrm sm}\to | L_g|^\mathrm{exp}$$ sending a stable map to its image. 
The points of $\overline{\V_{\delta}(S/J, L_g)}$ thus parametrize curves $X=X_0\cup X_1\cup\ldots\cup X_n$ in $|L_g|^\mathrm{exp}$ whose normalization outside of the nodes $n_i:=X_i\cap X_{i+1}$ is a nodal connected curve of arithmetic genus $\leq g-\delta$. 

Let $\M_{g-\delta}(S/J,L_g)^{\dagger}$ be as in the statement of Theorem \ref{expanded} and denote by $\tilde\M_{g-\delta}(S/J,L_g)^{\dagger}$ its semi-normalization, which also admits a map $$ \mu^\dagger:\tilde\M_{g-\delta}(S/J,L_g)^{\dagger}\to | L_g|^\mathrm{exp}$$ sending a stable map to its image. We may compose $\mu^\dagger$ with the natural map $|L_g|^\mathrm{exp}\to |L_g|$. By Theorem \ref{expanded}, the image of this composition is $\overline{V_{\delta}(S, L_g)}$ and this implies that the image of $ \mu^\dagger$ is again $\overline{\V_\delta(S/J,L_g)}$.

\begin{prop}\label{fabbrica}
Fix $g\geq 2$ and $0\leq\delta\leq g-1$. Then the following hold.
\begin{itemize}
\item[(i)] The stack $\overline{\V_{\delta}(S/J, L_g)}$ is connected and the same holds true for the relative normalization $\overline{\V_{\delta}(S/J, L_g)}^n$ of  $\overline{\V_{\delta}(S/J, L_g)}$ along  $\overline{\V_{\delta+1}(S/J, L_g)}$.
\item[(ii)] Let $\V$ and $\W$ be two intersecting irreducible component of $\overline{\V_{\delta}(S/J, L_g)}^n$ and let $\mathcal Z$ be a component of their intersection; then $\mathcal Z$ is generically reduced.
\end{itemize}
\end{prop}
\begin{proof}
The stack $\overline{\V_{\delta}(S/J, L_g)}$ coincides with the image of $ \mu^\dagger$ and is thus connected. We consider the divisor $\overline{\V_{\delta+1}(S/J, L_g)}\subset \overline{\V_{\delta}(S/J, L_g)}$. Since a general point $[X]\in\overline{\V_{\delta+1}(S/J, L_g)}$ parametrizes a nodal irreducible curve $X$ with $\delta+1$ nodes, then $ (\mu^\dagger)^{-1}([X])\subset  \tilde\M_{g-\delta}(S/J,L_g)^{sm}$ by the last statement in Theorem \ref{expanded}. Hence, $(\mu^\dagger)^{-1}([X])$ consists of the $\delta+1$ partial normalizations of $X$; in particular, $\tilde\M_{g-\delta}(S/J,L_g)^{\dagger}$ is generically smooth along $(\mu^{\dagger})^{-1}(\overline{\V_{\delta+1}(S/J, L_g)})$. Hence, $\mu^\dagger$ factors through a map $$\mu^n:\tilde\M_{g-\delta}(S/J,L_g)^{\dagger}\to \overline{\V_{\delta}(S/J, L_g)}^n$$ and point (i) directly follows from Theorem \ref{expanded}. 

Let now $\mathcal Z$ be as in point (ii).  Since normal singularities are unibranched (cf., e.g., \cite{Ko2}), then $\mathcal Z$ is not contained in the inverse image of $\overline{\V_{\delta+1}(S/J, L_g)}$ under the normalization map $\overline{\V_{\delta}(S/J, L_g)}^n\to \overline{\V_{\delta}(S/J, L_g)}$.

We first assume that a general point $\zeta$ of $\mathcal Z$ parametrizes an irreducible curve, so that locally around $\zeta$ the morphism 
$$\overline{\V_\delta(S/J,L_g)}\to \overline{V_{\delta}(S,L_g)}$$ (obtained as restriction of $|L_g|^\mathrm{exp}\to |L_g|$) is an isomorphism; the result thus follows from Proposition \ref{cohen}. 

We now treat the case where a general point $\zeta$ of $\mathcal Z$ parametrizes a curve $X=X_0\cup\ldots\cup X_n\in S[n]_0$ for some $n\geq 1$. More precisely, there exist $\underline h=(h_0,\ldots,h_n)\in \mathbb Z^{n+1}$ and $\underline g=(g_0,\ldots,g_n)\in\mathbb Z^{n+1}$ with $\sum_{i=0}^nh_i=g-\delta$, $\sum_{i=0}^ng_i=g$, $0\leq h_0\leq g_0$ and $1\leq h_i\leq g_i$ for $i>0$, such that $\mathcal Z$ is contained in the substack $\mathcal V({\underline h},{\underline g})$ of $\overline{\V_{\delta}(S/J, L_g)}$ parametrizing curves $$X=X_0\cup X_1\cup\ldots\cup X_n\subset S\cup R\cup\ldots\cup R=S[n]_0$$ such that $X_i$ has arithmetic genus $g_i$ and geometric genus $h_i$. The substack $\mathcal V({\underline h},{\underline g})$ lies in the intersection of $\overline{\V_{\delta}(S/J, L_g)}$ with $n$ Cartier divisors and can be identified with the open substack 
$$
U\subset \overline{V^{h_0}(S,L_{g_0})}\times \left[\overline{V^{h_1}(R,N_{1}(\underline g))}/\mathbb C^*\right]\times \cdots\times \left[\overline{V^{h_n}(R,N_{n}(\underline g))}/\mathbb C^*\right]
$$ parametrizing curves with no components in the singular locus of $S[n]_0$. Propositions \ref{dimension} and \ref{ruled} then imply that $\mathcal V({\underline h},{\underline g})$ has codimension $n$ in $\overline{\V_{\delta}(S/J, L_g)}$ and is generically reduced, thus proving that $\mathcal Z$ is generically reduced if $\mathrm{codim}\mathcal Z=n$. If instead $\mathrm{codim}\mathcal Z>n$, then $\V':=\V\cap \mathcal V({\underline h},{\underline g})$ and $\W':=\W\cap \mathcal V({\underline h},{\underline g})$ are unions of components of  $\mathcal V({\underline h},{\underline g})$ and $\mathcal Z$ is a component of the intersection $\V'\cap \W'$. Hence, $\V'$ and $\W'$ can be identified with open substacks of two products $\V'_0\times [\V_1'/\mathbb C^*]\times \cdots[\V_n'/\mathbb C^*]$ and $\W'_0\times [\W_1'/\mathbb C^*]\times \cdots[\W_n'/\mathbb C^*]$, where $\V'_0,\W_0'$ are components of $\overline{V^{h_0}(S,L_{g_0})}$, while $\V'_i,\W_i'$ are components of $\overline{V^{h_i}(R,N_{i}(\underline g))}$ for $i\geq1$. Hence, $\mathcal Z$ can be identified with an open subset of an irreducible component of $$(\V'_0\cap\W_0')\times [\V_1'\cap\W_1'/\mathbb C^*]\times \cdots[\V_n'\cap \W_n'/\mathbb C^*],$$ which is generically reduced again by Propositions \ref{cohen} and \ref{ruled} (using the fact that a categorical quotient of a generically reduced object is still generically reduced by the universal property).
\end{proof}
\color{black}

\section{Connectedness on a general $K3$ surface}\label{quattro}
In this section we will show how Theorems \ref{main} and \ref{expanded} imply analogous results on a general polarized $K3$ surface. 

Let $S$ and $S'$ be two surfaces both obtained as blow-ups of $\PP^2$ at two $9$-uples of general points, $p_1,\ldots,p_9$ and $p_1',\ldots, p_9'$, respectively. We assume that the anticanonical divisors on $S$ and $S'$ are represented by the same elliptic curve $J$ and that the relation $N_{J/S}\simeq N_{J/S'}^\vee$ holds. We glue $S$ and $S'$ along $J$ so that  $p_9'=p_{10}(g)$. Since $N_{J/S}\simeq \mathcal O_J(p_{10}(h)-p_{10}(h-1))$ for every $h\geq 0$, and the same holds for $S'$, our assumptions yield $p_{10}(h)=p_{10}(g-h)'$ for every $0\leq h\leq g$. In particular, all the pairs $(L_h,L_{g-h}')\in \Pic(S)\times \Pic(S')$ define equivalent polarizations on $Y_0:=S\cup_J S'$, as they differ only by the twist for a multiple of $(N_{J/S},N_{J/S'})$; we set $L:=[(L_h,L_{g-h}')]$.   The surface $Y_0$ is a stable $K3$ surface of type II and thus occurs as the central fiber of a flat map $$\chi:\mathcal Y\to \mathbb D$$ over a disc whose general fiber $Y_t$ is a smooth $K3$ surface of genus $g$ (cf. \cite[Prop. 2.5, Thm. 5.10]{Fr}). Furthermore, for every $0\leq h\leq g$ the family $\mathcal Y$ comes equipped with a relative line bundle $\mathcal L(h)$ restricting to the genus $g$ polarization $L_t$ on a general fiber $Y_t$ and to the polarization $(L_h,L_{g-h}')$ on $Y_0$. Since the line bundles $\mathcal L(h)$ only differ by a twist for a multiple of some component of the central fiber, they are all equivalent and we call $\mathcal L$ their class. We remark that this degeneration in the particular case where $S=S'$ and the points $p_1, \ldots, p_9\in \mathbb P^2$ are the base locus of a general pencil of plane cubics is the one used in \cite[\S 4]{MPT}.  

We denote by 
$$
\chi_\delta:\M_{g-\delta}(\mathcal Y^\mathrm{exp},\mathcal L)\to \mathbb D$$
the moduli stack of connected stable maps to expanded degenerations of $\chi$ constructed in \cite{Li1,Li2}. Over a point $t\in \mathbb D^*=\mathbb D\setminus\{ 0\}$, the fiber of $\chi_\delta$ is simply the moduli stack $\M_{g-\delta}(Y_t,L_t)$ of ordinary stable maps on $Y_t$, while the fiber over $0$ is the stack $\M_{g-\delta}(Y_0^\mathrm{exp},L)$ parametrizing stable maps (in the sense of J. Li) to some expanded target degeneration of $Y_0$ of the form:
$$
Y_0[n]_0:=S\cup_J R\cup_J\ldots\cup_JR\cup_JS'$$
for some $n\geq 0$. As already used in \cite{MPT}, a stable map to such an expanded degeneration can be split in a non-unique way into relative stable maps to $(S,J)$ and $(S',J)$. In particular, $\M_{g-\delta}(Y_0^\mathrm{exp},L)$ can be written as a non-disjoint union
\begin{equation}\label{cor}\M_{g-\delta}(Y_0^\mathrm{exp},L)=\bigcup_{\substack{g_1+g_2=g\\h_1+h_2=g-\delta}}\M_{h_1}(S/J,L_{g_1})\times \M_{h_2}(S'/J,L_{g_2}'),\end{equation}
where each factor in the above decomposition can be realized as a Cartier divisor on $\M_{g-\delta}(\mathcal Y^\mathrm{exp},\mathcal L)$. Let $$\M_{g-\delta}(\mathcal Y^\mathrm{exp},\mathcal L)^{sm}\to \mathbb D$$ be the substack of $\M_{g-\delta}(\mathcal Y^\mathrm{exp},\mathcal L)$ whose fiber over $t\in \mathbb D^*$ is the substack $\M_{g-\delta}(Y_t,L_t)^{sm}$ of $\M_{g-\delta}(Y_t,L_t)$ parametrizing smoothable stable maps, and set $\M_{g-\delta}(Y_0^\mathrm{exp},L)^{sm}:=\M_{g-\delta}(\mathcal Y^\mathrm{exp},\mathcal L)^{sm}\times_{\mathbb D}0$. 

Similarly, we denote by $$e_\delta: |\mathcal L|^{\mathrm{exp}}\to \mathbb D$$ the good degeneration of the relative linear system $|\mathcal L|^*\to\mathbb D^*$, which is a particular case of good degenerations of relative Hilbert schemes introduced and studied in \cite{LW}. The space $|\mathcal L|^{\mathrm{exp}}$ is a Deligne-Mumford stack, separated, proper over $\mathbb D$ and of finite type. A fiber of $e_\delta$ over a general $t\in \mathbb D$ is isomorphic to the linear system $|L_t|$ on $Y_t$. Points of the central fiber, that we denote by $|L|^{\mathrm{exp}}$, parametrize equivalence classes of curves $X=X_0\cup X_1\cup\ldots\cup X_n\cup X_0'$ in some expanded target degenerations $Y_0[n]_0$ of $Y_0$ with no components in the singular locus of $Y_0[n]_0$. Since Severi varieties may be defined functorially, for any fixed $0\leq \delta\leq g$ there is a $\chi$-relative Severi variety $s_\delta:\V_{\delta}(\Y, \mathcal L)^*\to \mathbb D^*$ such that the fiber over $t\in \mathbb D^*$ is the Severi variety $V_\delta(Y_t,L_t)$.  We denote by $\overline{\V_\delta(\mathcal Y^{\mathrm{exp}},\mathcal L)}$ the closure of $\V_{\delta}(\Y, \mathcal L)^*$ in $ |\mathcal L|^{\mathrm{exp}}$ and by 
$$\overline s_{\delta}:\overline{\V_\delta(\mathcal Y^{\mathrm{exp}},\mathcal L)} \to\mathbb D$$ the natural morphism. The stack $\overline{\V_\delta(\mathcal Y^{\mathrm{exp}},\mathcal L)}$ can be alternatively realized as the image of the natural map 
$$
\M_{g-\delta}(\mathcal Y^\mathrm{exp},\mathcal L)^{sm}\to |\mathcal L|^{\mathrm{exp}}.
$$
The analougue of \eqref{cor} for the central fiber $\overline{\mathcal V_\delta(Y_0^\mathrm{exp},L)}:=\overline{\V_\delta(\mathcal Y^{\mathrm{exp}},\mathcal L)}\times_{\mathbb D}0$ of $\overline s_{\delta}$ is then stated in the following result.
\begin{lem}\label{tosto}
The stack $\overline{\mathcal V_\delta(Y_0^\mathrm{exp},L)}$ decomposes in the following non-disjoint union:
\begin{equation}\label{mf}
\overline{\mathcal V_\delta(Y_0^\mathrm{exp},L)}=\bigcup_{\substack{g_1+g_2=g\\\delta_1+\delta_2=\delta}}\overline{\V_{\delta_1}(S/J,L_{g_1})}\times \overline{\V_{\delta_2}(S'/J,L_{g_2}')}.
\end{equation}
Each factor in the decomposition \eqref{mf} appears with multiplicity $1$ and defines a Cartier divisor in $\overline{\V_\delta(\mathcal Y^{\mathrm{exp}},\mathcal L)}$.
\end{lem}
\begin{proof}
As concerns the inclusion $\supset$, we recall that a general point in any irreducible component of $\overline{\V_{\delta_1}(S/J,L_{g_1})}$ (respectively, $\overline{\V_{\delta_2}(S'/J,L_{g_2}')}$) parametrizes an irreducible curve $C\in V_{\delta_1}(S,L_{g_1})$ (respectively, $C'\in V_{\delta_2}(S',L_{g_2}')$). By gluing $C$ and $C'$ along their intersection point $p_{10}(g_1)=p_{10}(g_2)'$ with $J$, one obtains a curve $X=C\cup C'\subset Y_0=S\cup S'$ with $\delta_1+\delta_2$ nodes outside of $X$; there is no obstruction to deforming such an $X$ outside of the central fiber of $\overline s_\delta$ and this proves $\supset$.

We now prove the opposite inclusion $\subset$. Let $\V$ be a component of the central fiber $\overline{\mathcal V_\delta(Y_0^\mathrm{exp},L)}$ of $\overline s_\delta$. Then $\V$ has dimension $g-\delta$ by upper semicontinuity along with the fact that $\mathbb D$ is $1$-dimensional. A general point of $\V$ parametrizes a curve $X\subset Y_0[n]_0$ for some $n\geq 0$ such that $X\in|\mathcal L|^{\mathrm{exp}}$  and the normalization of $X$ outside of the singular locus of $Y_0[n]_0$ has arithmetic genus $h\leq g-\delta$. Propositions \ref{dimension} and \ref{ruled} then yield $n=0$ and $h=g-\delta$, so that $X=C\cup C'$ with $C\in\overline{V^{g_1-\delta_1}(S,L_{g_1})} =\overline{V_{\delta_1}(S,L_{g_1})}$ and $C'\in\overline{V^{g_2-\delta_2}(S',L_{g_2})} =\overline{V_{\delta_2}(S',L_{g_2}')}$ for some integers $0\leq \delta_1\leq g_1$, $0\leq \delta_2\leq g_2$ such that $g_1+g_2=g$ and $\delta_1+\delta_2=\delta$. This proves $\subset$.

The last part of the statement is a consequence of the same property for the decomposition \eqref{cor}, that follows from \cite{Li2} since $L_h\cdot J=1$ for every $h\geq 0$. 
\end{proof}

\color{black}
\begin{prop}\label{paths}
If $0\leq\delta\leq g-1$, every component $\V$ of $\overline{\mathcal V_\delta(Y_0^\mathrm{exp},L)}$ can be connected to $|L_{0}|\times \overline{\V_{g-\delta}(S'/J,L_{g}')}$ through a sequence of irreducible components $$\V=\V_0,\V_1,\ldots, \V_r\subset |L_{0}|\times \overline{\V_{g-\delta}(S'/J,L_{g}')}$$ such that for all $0\leq i\leq r-1$ the following hold:
\begin{itemize}
\item[(i)] the intersection $\V_i\cap \V_{i+1}$ has codimension $1$;
\item[(ii)] a general point of $\V_i\cap \V_{i+1}$ parametrizes a curve $$X=X_0\cup X_1\cup X_0'\subset Y_0[1]_0$$ such that the components $X_0\subset S$ and $X_0'\subset S'$ are nodal,  the component $X_1\subset R$ is immersed, the normalization of $X$ outside of its intersection points with the singular locus of $ Y_0[1]_0$ has arithmetic genus precisely $g-\delta$. 
\end{itemize}
In particular, the intersection $\V_i\cap \V_{i+1}$ is generically reduced.
\end{prop}
\begin{proof}
By Lemma \ref{tosto}, there exist integers $0\leq \delta_1\leq g_1$, $0\leq \delta_2\leq g_2$ such that $g_1+g_2=g$ and $\delta_1+\delta_2=\delta$ and $\V_0:=\V=\W_0\times \W_0'$ for some irreducible components $\W_0$ of $\overline{\V_{\delta_1}(S/J,L_{g_1})}$ and $\W_0'$ of $ \overline{\V_{\delta_2}(S'/J,L_{g_2}')}$. 

If $g_1-\delta_1\geq 1$, then $\W_0$ contains in codimension $1$ points that parametrize curves $C=C_0\cup C_1\subset S[1]_0$ with $C_0\in V_{\delta_1}(S,L_{g_1-1})$ and $C_1\simeq J\in |J_0+f_{p_{10}(g_1)}|$. 
 For any $C_0'\in V_{\delta_2}(S',L_{g_2}')$ the nodal curve $C:=C_0\cup C_1\cup C_0'$ also defines a point of $\overline{\V_{\delta_1}(S/J,L_{g_1-1})}\times \overline{\V_{\delta_2}(S'/J,L_{g_2+1}')}$ and this proves that $\V_0:=\W_0\times \W_0'$ can be connected to a component $$\V_1=\W_1\times \W_1'\subset \overline{\V_{\delta_1}(S/J,L_{g_1-1})}\times \overline{\V_{\delta_2}(S'/J,L_{g_2+1}')}$$ so that the intersection $\V_0\cap \V_1$ satisfies conditions (i)-(ii) in the statement. By repeating the same argument $g_1-\delta_1$ times, we find a sequence
$$\V=\V_0,\V_1,\ldots, \V_{g_1-\delta_1}$$
of irreducible components of $\overline{\mathcal V_\delta(Y_0^\mathrm{exp},L)}$ such that $\V_i=\W_i\times \W_i'$ is an irreducible component of $\overline{\V_{\delta_1}(S/J,L_{g_1-i})}\times \overline{\V_{\delta_2}(S'/J,L_{g_2+i}')}$ and that conditions (i)-(ii) in the statement hold for the intersection $\V_i\cap \V_{i+1}$.
 
We now start from $$\V_{g_1-\delta_1}=\W_{g_1-\delta_1}\times \W_{g_1-\delta_1}'\subset \overline{\V_{\delta_1}(S/J,L_{\delta_1})}\times \overline{\V_{\delta_2}(S'/J,L_{g-\delta_1}')}$$ and notice that $\W_{g_1-\delta_1}'\subset \overline{\V_{\delta_2}(S'/J,L_{g-\delta_1}')}$ contains in codimension $1$ points that parametrize curves $D=\tilde J\cup D_0'\subset S[1]_0'$, where $D_0'\in |L_{g-\delta-1}'|$ and $\tilde J$ is an irreducible elliptic curve such that $\tilde J\in V^1(R,(\delta_2+1)J_0+f_{p_{10}(\delta+1)})$. For every rational curve $D_0\in V_{\delta_1}(S,L_{\delta_1})$ the curve $D:=D_0\cup \tilde J\cup D_0'$ also defines a point of $\overline{\V_{\delta}(S/J,L_{\delta+1})}\times |L_{g-\delta-1}'|^{\mathrm{exp}}$; this proves that $\V_{g_1-\delta_1}$ can be connected to a component $$\V_{g_1-\delta_1+1}:=\W_{g_1-\delta_1+1}\times |L_{g-\delta-1}'|^{\mathrm{exp}}\subset \overline{\V_{\delta}(S/J,L_{\delta+1})}\times |L_{g-\delta-1}'|^{\mathrm{exp}}$$ so that (i) and (ii) hold.

Finally, we use the fact that the component $\W_{g_1-\delta_1+1}\subset\overline{\V_{\delta}(S/J,L_{\delta+1})}$ contains in codimension $1$ curves of the form $E_9+\overline J$, where $E_9$ is the ninth exceptional divisor on $S$ (and thus the only curve in the linear system $|L_0|$) and $\overline J$ is an irreducible curve such that $\overline J\in V^1(R, (\delta+1)J_0+f_{p_{10}(\delta+1)})$. For any curve $F_0'\in |L_{g-\delta-1}'|$, the divisor $E_9+\overline J+F_0'\in Y_0[1]_0$ also defines a point of $|L_0|\times \overline{\V_\delta(S'/J,L_g')}$. This proves that $\V_{g_1-\delta_1+1}$ is connected to a component $\V_{g_1-\delta_1+2}$ of $|L_0|\times \overline{\V_\delta(S'/J,L_g')}$ so that (i)-(ii) hold for $\V_{g_1-\delta_1+1}\cap\V_{g_1-\delta_1+2}$. 

It only remains to prove that conditions (i)-(ii) imply that the intersection $\V_i\cap \V_{i+1}$ is generically reduced. Conditions (i)-(ii) together with Propositions \ref{dimension} and \ref{ruled} yield that every component of $\V_i\cap \V_{i+1}$ is birational to an open substack of 
$$
\overline{V_{\delta_0}(S,L_{g_0})}\times \left[\overline{V^{h_1}(R,g_1J_0+ f_{p_{10}(g_0+g_1)})}/\mathbb C^*\right]\times \overline{V_{\delta_0'}(S',L_{g_0'}')}
$$
for some integers $0\leq \delta_0\leq g_0$, $0\leq \delta_0'\leq g_0'$, $1\leq h_1\leq g_1$ such that $g_0+g_1+g_0'=g$ and $\delta_0+g_1-h_1+\delta_0'=\delta$. The generic reducedness of $\V_i\cap \V_{i+1}$ thus follows from Propositions \ref{dimension} and \ref{ruled}.\end{proof}
\begin{thm}\label{connected}
Let $(Y,L)$ be a general primitively polarized $K3$ surface of genus $g\geq 2$ and fix $0\leq \delta\leq g-1$. Then the closure of the Severi variety $\overline{V_{\delta}(Y,L)}\subset |L|$ is connected and the same holds true for the relative normalization $\overline{V_{\delta}(Y,L)}^n$ of $\overline{V_{\delta}(Y,L)}$ along $\overline{V_{\delta+1}(Y,L)}$. 
\end{thm}
\begin{proof} 
Let $\overline s_{\delta}:\overline{\V_\delta(\mathcal Y^{\mathrm{exp}},\mathcal L)} \to\mathbb D$ be the good degeneration of the relative Severi variety to the family $\chi:\mathcal Y\to \mathbb D$ as at the beginning of this section. We denote by $ s_{\delta}^n:\overline{\V_\delta(\mathcal Y^{\mathrm{exp}},\mathcal L)}^n \to\mathbb D$ the relative normalization of $\overline{\V_\delta(\mathcal Y^{\mathrm{exp}},\mathcal L)}$ along $\overline{\V_{\delta+1}(\mathcal Y^{\mathrm{exp}},\mathcal L)}$. In particular, a general fiber of $ s_{\delta}^n$ is the normalization of $\overline{V_\delta(Y_t,L_t)}$ along $\overline{V_{\delta+1}(Y_t,L_t)}$, while the central fiber is the normalization of $\overline{\mathcal V_\delta(Y_0^\mathrm{exp},L)}$ along $\overline{\mathcal V_{\delta+1}(Y_0^\mathrm{exp},L)}$. We need to show that a general fiber of $ s_{\delta}^n$ is connected.

First of all, we note that the central fiber $\overline{\mathcal V_\delta(Y_0^\mathrm{exp},L)}$ of $\overline s_{\delta}$ is generically reduced. Indeed, this follows directly from Lemma \ref{tosto} and Proposition \ref{dimension}(iii). Obviously, the same holds true for the central fiber of $ s_{\delta}^n$, thus implying that two components of a general fiber of $ s_{\delta}^n$ remain distinct also on the central fiber. 

By Proposition \ref{paths}, every component of the central fiber $\overline{\mathcal V_\delta(Y_0^\mathrm{exp},L)}$ of $\overline s_{\delta}$ can be connected to $|L_{0}|\times \overline{\V_{g-\delta}(S'/J,L_{g}')}$ through a sequence of irreducible components $\V=\V_0,\V_1,\ldots, \V_r\subset |L_{0}|\times \overline{\V_{g-\delta}(S'/J,L_{g}')}$ such that the intersections $\V_i\cap\V_{i+1}$ are generically reduced and not contained in $\overline{\mathcal V_{\delta+1}(Y_0^\mathrm{exp},L)}$. In particular, the latter property also implies that every component $\V^n$ of the relative normalization $\overline{\mathcal V_\delta(Y_0^\mathrm{exp},L)}^n$ of $s_{\delta}^n$ can be connected to $|L_{0}|\times \overline{\V_{g-\delta}(S'/J,L_{g}')}^n$  through a sequence of irreducible components $\V^n=\V_0^n,\V_1^n,\ldots, \V_r^n\subset |L_{0}|\times \overline{\V_{g-\delta}(S'/J,L_{g}')}^n$ such that the intersections $\V_i^n\cap\V_{i+1}^n$ are generically reduced. Since $|L_0|$ is a point, Proposition \ref{fabbrica} implies that $|L_{0}|\times \overline{\V_{g-\delta}(S'/J,L_{g}')}^n$ is connected and that the intersection of two components of $|L_{0}|\times \overline{\V_{g-\delta}(S'/J,L_{g}')}^n$ is generically reduced. 

We use this in order to conclude that a general fiber of $ s_{\delta}^n$ is connected, too. If these were not the case, there would exist two irreducible components $\Z_1$ and $\Z_2$ of $\overline{\V_\delta(\mathcal Y^{\mathrm{exp}},\mathcal L)}^n$ (whose restriction to the central fiber of $ s_{\delta}^n$ we denote by $Z_1$ and $Z_2$, respectively) such that $\Z_1\cap\Z_2=Z_1\cap Z_2$; in particular, a general point of $Z_1\cap Z_2$ should be a nonreduced point of the central fiber by, e.g., \cite[76.36.8]{St}. However, we may assume that this is not the case because the above discussion yields that any two components of the central fiber can be connected through a sequence of components whose intersection is generically reduced. 
\end{proof}

\section{Irreducibility on a general $K3$ surface}\label{cinque}
In this section, we will focus on the Severi problem for general polarized $K3$ surfaces. As we have already proved that Severi varieties of positive dimension are connected, the irreducibility problem can be approached by investigating how two irreducible components may intersect.

\begin{prop}\label{two}
Let $(Y,L)$ be a general primitively polarized $K3$ surface of genus $g\geq 2$ and fix $0\leq \delta\leq g-1$. The intersection of two irreducible components of $\overline{V_\delta(Y,L)}$, if nonempty, has pure codimension $1$. 
\end{prop}
\begin{proof}
Since $(Y,L)$ is general, we may assume that all curves in $|L|$ are integral and that $\overline{V^\delta(Y,L)}=\overline{V_{g-\delta}(Y,L)}$. The proof proceeds as the one of Proposition \ref{cohen} and is actually easier because all curves in $|L|$ are integral. In this case $I=D(\widetilde{\phi})$ and the same proof as that of Lemma \ref{lemcz} implies that the locus in $I$ where the fibers of the projection $t: I\longrightarrow |L|$ have positive dimension has dimension $\leq g-\delta-2$.
\end{proof}

\begin{thm}\label{local}
Let $(Y,L)$ be a general primitively polarized $K3$ surface of genus $g\geq 4$ and fix $0\leq \delta \leq g-4$. Then, the Severi variety $V_{\delta}(Y,L)$ is irreducible.
\end{thm}
\begin{proof}
Since $(Y,L)$ is general, we may assume that all curves in $|L|$ are integral. By Theorem \ref{connected}, $\overline{V_{\delta}(Y,L)}$ and its relative normalization $\overline{V_{\delta}(Y,L)}^n$ along $\overline{V_{\delta+1}(Y,L)}$ are connected. If reducible, $\overline{V_{\delta}(Y,L)}$ contains two irreducible components $V,W$ whose intersection $V\cap W$ is nonempty and thus of codimension $1$ by Proposition \ref{two}. Since $\overline{V_{\delta}(Y,L)}^n$ is still connected, we may further assume that $V\cap W$ is not contained in $\overline{V_{\delta+1}(Y,L)}$. It is therefore enough to show that no codimension $1$ component of the singular locus of $\overline{V_{\delta}(Y,L)}$ may contain such an intersection.

Let $Z$ be a component of $\mathrm{Sing}\overline{V_{\delta}(Y,L)}$ such that $Z$ is not contained in $\overline{V_{\delta+1}(Y,L)}$ and $Z$ has codimension $1$. Let $C\in Z$ be a general point and denote by $f$ the composition of the normalization map $\nu:\tilde C\to C$ with the inclusion of $C$ in $Y$. Since $Z$ has dimension $g-\delta-1$, Theorem \ref{dense} and Proposition \ref{known} imply that $\tilde C$ is a smooth irreducible curve of genus either $g-\delta$ or $g-\delta-1$, and the latter case does not occur because $Z$ is not contained in $\overline{V_{\delta+1}(Y,L)}$  .

Hence, $\tilde C$ has genus $g-\delta$ and, by generality, we can assume that all points in a dense open subset of $Z$ parametrize curves with the same singularities as $C$. We may thus apply \cite[p. 26]{AC} as in the proof of Proposition \ref{dimension} to obtain
\begin{equation}\label{atene}
g-\delta-1=\dim Z\leq h^0(\overline N_f),
\end{equation}
where $\overline N_f\simeq \omega_{\tilde C}(-R)$ with $R$ being the ramification divisor of $f$. Inequality \eqref{atene} then yields $\deg R\leq 2$. 

If $\deg R=0$, then $f$ is unramified and $N_f=\ov N_f=\omega_{\tilde C}$,  If $\deg R=1$, then $C$ has only one ordinary cusp and, denoting by $Q$ the point of $\tilde C$ mapping to it, we have $N_f=\omega_{\tilde C}(-Q)\oplus \mathcal O_Q$. In both cases one has $h^0(N_f)=g-\delta$ and thus $f$ defines a smooth point of the moduli space of genus $g-\delta$ stable maps $M_{g-\delta}(Y,L)$, and more precisely of the locus $M_{g-\delta}(Y,L)^{\mathrm{sm}}$ parametrizing smoothable stable maps. Let $\mu$ be the morphism from the semi-normalization of $M_{g-\delta}(Y,L)^{\mathrm{sm}}$ to $\overline{V^{g-\delta}(Y,L)}=\overline{V_\delta(Y,L)}$. Locally around $f$ the morphism $\mu$ is injective: indeed, the inverse image under $\mu$ of an irreducible curve of geometric genus exactly $g-\delta$ is the only point defined by the composition of its normalization map with its inclusion in $S$. Therefore, the smoothness of  $M_{g-\delta}(Y,L)^{\mathrm{sm}}$ at $f$ yields that $\overline{V_\delta(S,L)}$, if singular at $C$, has a unibranched singularity there. In particular, $W$ cannot lie in the intersection of two irreducible components of $\overline{V_\delta(S,L)}$.

We now treat the remaining case $\deg R=2$, where \eqref{atene} implies that  $\tilde C$ is hyperelliptic and $R$ is a divisor in the $g^1_2$. By \cite[Rmk. 5.6]{KLM}, curves in $|L|$ with hyperelliptic normalization of any fixed geometric genus $\geq 2$ move in dimension $2$ and hence $g-\delta-1=\dim W\leq 2$ yielding a contradiction as soon as $\delta\leq g-4$. 
\end{proof}

\end{document}